\documentclass[a4paper,11pt]{article}

\usepackage[english]{babel}
\usepackage{amssymb}
\usepackage{amsmath,amsfonts,latexsym}
\usepackage{amsthm}
\usepackage{graphicx}
\usepackage{fancyhdr}
\usepackage{pst-all}
\usepackage[T1]{fontenc}
\usepackage{hyperref}

\newcommand{\cGuv}{$\chi(G\backslash\{u,v\})\geq\chi(G)-1$}

\newcommand{\co}{$\chi(G)>\omega(G)$}

\def\d{\hbox{-}}

\def\l{,\ldots,}

\newtheorem{conjecture}{Conjecture}

\newtheorem{theorem}{}[section]
\newtheorem{lemma}[theorem]{}

\newtheorem{corollary}[theorem]{}
\newtheoremstyle{styleclaim}{}{}{\itshape}{}{}{}{ }{(\thmnumber{#2})}

\theoremstyle{styleclaim}
\newcounter{counter_claim}[theorem]
\newtheorem{claim}[counter_claim]{}
\newcommand{\cref}[1]{(\ref{#1})}

\begin{document}
\title{On the Erd\"os-Lov\'asz Tihany Conjecture for Claw-Free Graphs}
\date{\today}
\author{Maria Chudnovsky\thanks{Columbia University, New York, NY 10027, USA. E-mail: mchudnov@columbia.edu. Partially supported by NSF grants IIS-1117631 and DMS-1001091.} \and Alexandra Fradkin\thanks{Center for Communications Research, Princeton, NJ 08540, USA.} \and Matthieu Plumettaz\thanks{Columbia University, New York, NY 10027, USA. E-mail: mp2761@columbia.edu. Partially supported by NSF grant DMS-1001091.}}

\maketitle

\begin{abstract}
In 1968, Erdos and Lovasz conjectured that for every graph $G$ and all integers $s,t\geq 2$ such that $s+t-1=\chi(G) > \omega(G)$, there exists a partition $(S,T)$ of the vertex set of $G$ such that $\chi(G|S)\geq s$ and $\chi(G|T)\geq t$. For general graphs, the only settled cases of the conjecture are when $s$ and $t$ are small. Recently, the conjecture was proved for a few special classes of graphs: graphs with stability number $2$~\cite{quasi-line}, line graphs~\cite{line} and quasi-line graphs~\cite{quasi-line}. In this paper, we consider the conjecture for claw-free graphs and present some progress on it.
\end{abstract}

\section{Introduction}
In 1968, Erd\H{o}s and Lov\'asz made the following conjecture:

\begin{conjecture}[Erd\"os-Lov\'asz Tihany]\label{ELT}
For every graph $G$ with $\chi(G) > \omega(G)$ and any two integers $s,t \geq 2$ with $s+t=\chi(G)+1$, there is a partition $(S,T)$ of the vertex set $V(G)$ such that $\chi(G|S)\geq s$ and $\chi(G|T)\geq t$.
\end{conjecture}

Currently, the only settled cases of the conjecture are $(s,t) \in \{(2,2), (2,3), (2,4), $ $ (3,3), (3,4), (3,5)\}$~\cite{BrownJung, Mozhan, StiebitzK5, Stiebitz}. Recently, Balogh et. al. proved Conjecture~\ref{ELT} for the class of graphs known as \emph{quasi-line graphs} (a graph is a quasi-line graph if for every vertex $v$, the set of neighbors of $v$ can be expressed as the union of two cliques).  In this paper we consider a class of graphs that is a proper superset of the class of quasi-line graphs: \emph{claw-free graphs}. We prove a weakened version of Conjecture~\ref{ELT} for this class of graphs.  Before we state our main result we need to set up some notation and definitions.

In this paper all graphs are finite and simple. Given a graph $G$, let $V(G)$, $E(G)$ denote the set of vertices and edges of $G$, respectively.  A $k$-coloring of $G$ is a map $c:V(G) \rightarrow \{1,\dots,k\}$ such that for every pair of adjacent vertices $v,w \in V(G)$, $c(v) \neq c(w)$.  We may also refer to a $k$-coloring simply as a coloring.  The \emph{chromatic number} of $G$, denoted by $\chi(G)$, is the smallest integer such that there is a $\chi(G)$-coloring of $G$.

A \emph{clique} in $G$ is a set of vertices of $G$ that are all pairwise adjacent.  A \emph{stable set} in $G$ is a set of vertices that are all pairwise non-adjacent.  A set $S\subseteq V(G)$ is an \textit{anti-matching} if every vertex in $S$ is non-adjacent to at most one vertex of $S$.  A \emph{brace} is a clique of size 2, a \emph{triangle} is a clique of size 3 and a \emph{triad} is a stable set of size 3.  The \emph{clique number} of $G$, denoted by $\omega(G)$, is the size of a maximum clique in $G$, and the \emph{stability number} of $G$, denoted by $\alpha(G)$ is the size of the maximum stable set in $G$.

Let $G$ be a graph such that \co.  We say that a brace $\{u,v\}$ is \textit{Tihany} if \cGuv.  More generally, if $K$ is a clique of size $k$ in $G$, then we say that $K$ is \textit{Tihany} if $\chi(G \setminus K) \geq \chi(G)-k+1$. 

Let $K$ be a clique in $G$. We denote by $C(K)$ the set of common neighbors of the members of $K$, by $A(K)$ the set of their common non-neighbors, and by $M(K)$ the set of vertices that are mixed on the clique $K$.  Formally, 
\begin{align*}
C(K) =&\{x\in V(G)\backslash K: ux\in E\makebox{ for all }u\in V(K)\}\\
A(K) =&\{x:ux\notin E\makebox{ for all }x \in K\}\\ 
 M(K) =&V(G) \setminus (C(K) \cup A(K)).
\end{align*}
 We say that a clique $K$ is \emph{dense} if $C(K)$ is a clique and we say that it is \emph{good} if $C(K)$ is an anti-matching.

The following theorem is the main result of this paper:

\begin{theorem}\label{main}
Let $G$ be a claw-free graph with \co.  Then there exists a clique $K$ with $|K|\leq 5$ such that $\chi(G \backslash K) > \chi(G) - |K|$.
\end{theorem}

To prove~\ref{main} we use a structure theorem for claw-free graphs (due to the first author and Seymour) that appears in~\cite{claws5} and is described in the next section.  Section 3 contains some lemmas that serve as ''tools'' for later proofs.  The next 6 sections are devoted to dealing with the different outcomes of the structure theorem, proving that a (subgraph) minimal counterexample to~\ref{main} does not fall into any of those classes.  Finally, in section~\ref{sec:main_proof} all of these results are collected to prove~\ref{main}.
\section{Structure Theorem}
The goal of this section is to state and describe the structure theorem for claw-free graphs appearing in \cite{claws5} (or, more precisely, its corollary).  We begin with some definitions which are modified from \cite{claws5}.

Let $X,Y$ be two subsets of $V(G)$ with $X \cap Y = \emptyset$.  We say that
$X$ and $Y$ are \emph{complete} to each other if every vertex of $X$ is
adjacent to every vertex of $Y$, and we say that they are \emph{anticomplete} 
to each other
if no vertex of $X$ is adjacent to a member of $Y$.  Similarly, if $A \subseteq
V(G)$ and $v \in V(G) \setminus A$, then $v$ is \emph{complete} to $A$ if $v$ is
adjacent to every vertex in $A$, and \emph{anticomplete} to $A$ if $v$ has no
neighbor in $A$. 

Let $F \subseteq V(G)^2$ be a set of unordered pairs of distinct vertices of $G$ such that every vertex appears in at most one pair.  Then $H$ is a \emph{thickening} of $(G,F)$ if for every $v \in V(G)$ there is a nonempty subset $X_v \subseteq V(H)$, all pairwise disjoint and with union $V(H)$ satisfying the following:
\begin{itemize}
\item for each $v \in V(G)$, $X_v$ is a clique of $H$
\item if $u,v \in V(G)$ are adjacent in $G$ and $\{u,v\} \not\in F$, then $X_u$ is complete to $X_v$ in $H$
\item if $u,v \in V(G)$ are nonadjacent in $G$ and $\{u,v\} \not\in F$, then $X_u$ is anticomplete to $X_v$ in $H$
\item if $\{u,v\} \in F$ then $X_u$ is neither complete nor anticomplete to $X_v$ in $H$.
\end{itemize}

\noindent Here are some classes of claw-free graphs that come up in the structure theorem.
\begin{itemize}
\item {\bf Graphs from the icosahedron.} The \emph{icosahedron} is the unique planar graph with twelve vertices all of degree five.  Let it have vertices $v_0,v_1,\ldots,v_{11}$, where for $1 \leq i \leq 10$, $v_i$ is adjacent to $v_{i+1},v_{i+2}$ (reading subscripts modulo 10), and $v_0$ is adjacent to $v_1,v_3,v_5,v_7,v_9$, and $v_{11}$ is adjacent to $v_2,v_4,v_6,v_8,v_{10}$.  Let this graph be $G_0$.  Let $G_1$ be obtained from $G_0$ by deleting $v_{11}$ and let $G_2$ be obtained from $G_1$ by deleting $v_{10}$.  Furthermore, let $F'=\{\{v_1,v_4\},\{v_6,v_9\}\}$ and let $F \subseteq F'$.  
    
    Let $G \in \mathcal{T}_1$ if $G$ is a thickening of $(G_0,\emptyset)$, $(G_1,\emptyset)$, or $(G_2,F)$ for some $F$.
    
\item {\bf Fuzzy long circular interval graphs.}
Let $\Sigma$ be a circle, and let $F_1,\ldots,F_k \subseteq \Sigma$ be homeomorphic to the interval $[0,1]$, such that no two of $F_1,\ldots,F_k$ share an 
endpoint, and no three of them have union $\Sigma$.  Now let $V \subseteq \Sigma$ be finite, and let $H$ be a graph with vertex set $V$ in which distinct $u,v \in V$ are adjacent precisely if $u,v \in F_i$ for some $i$. 

Let $F' \subseteq V(H)^2$ be the set of pairs $\{u,v\}$ such that $u,v$ are distinct endpoints of $F_i$ for some $i$.  Let $F \subseteq F'$ such that every vertex of $G$ appears in at most one member of $F$.  Then $G$ is a \emph{fuzzy long circular interval graph} if for some such $H$ and $F$, $G$ is a thickening of $(H,F)$.  

Let $G \in \mathcal{T}_2$ if $G$ is a fuzzy long circular interval graph. 

\item {\bf Fuzzy antiprismatic graphs.}
A graph $K$ is \emph{antiprismatic} if for every $X \subseteq V(K)$ with 
$|X|=4$, $X$ is not a claw and there are at least two pairs of vertices in $X$ that are adjacent.
Let $H$ be a graph and let $F \subseteq V(H)^2$ be a set of pairs $\{u,v\}$ 
such that every vertex of $H$ is in at most one member of $F$ and
\begin{itemize}
\item no triad of $H$ contains $u$ and no triad of $H$ contains $v$, or
\item there is a triad of $H$ containing both $u$ and $v$ and no other triad of $H$ contains $u$ or $v$.
\end{itemize}
Thus $F$ is the set of ``changeable edges'' discussed in \cite{claws1}.
The pair $(H,F)$ is {\em antiprismatic} if for every $F' \subseteq F$, 
the graph obtained from $H$ by changing  the adjacency of all the vertex pairs
in $F'$ is antiprismatic.
We say that a graph $G$ is a \emph{fuzzy antiprismatic graph} if $G$ is a 
thickening of $(H,F)$ for some antiprismatic pair $(H,F)$.

Let $G \in \mathcal{T}_3$ if $G$ is a fuzzy antiprismatic graph.

\end{itemize}

Next, we define what it means for a claw-free graph to admit a ``strip-structure''.

A \emph{hypergraph} $H$ consists of a finite set $V(H)$, a finite set $E(H)$, and an incidence relation between $V(H)$ and $E(H)$ (that is, a subset of $V(H) \times E(H)$).  For the statement of the structure theorem, we only need hypergraphs such that every member of $E(H)$ is incident with either one or two members of $V(H)$ (thus, these hypergraphs are graphs if we allow ``graphs'' to have loops and parallel edges).  For $F \in E(H)$, $\overline{F}$ denotes the set of all $h \in V(H)$ incident with $F$.

Let $G$ be a graph.  A \emph{strip-structure} $(H,\eta)$ of $G$ consists of a hypergraph $H$ with $E(H) \neq \emptyset$, and a function $\eta$ mapping each $F \in E(H)$ to a subset $\eta(F)$ of $V(G)$, and mapping each pair $(F,h)$ with $F \in E(H)$ and $h \in \overline{F}$ to a subset $\eta(F,h)$ of $\eta(F)$, satisfying the following conditions. \\ 

\noindent {\bf (SD1)} The sets $\eta(F)$ ($F \in E(H)$) are nonempty and pairwise disjoint and have union $V(G)$. \\ 

\noindent {\bf (SD2)} For each $h \in V(H)$, the union of the sets $\eta(F,h)$ for all $F \in E(H)$ with $h \in \overline{F}$ is a clique of $G$. \\ 

\noindent {\bf (SD3)} For all distinct $F_1,F_2 \in E(H)$, if $v_1 \in \eta(F_1)$ and $v_2 \in \eta(F_2)$ are adjacent in $G$, then there exists $h \in \overline{F_1} \cap \overline{F_2}$ such that $v_1 \in \eta(F_1,h)$ and $v_2 \in \eta(F_2,h)$. \\ 

There is also a fourth condition, but it is technical and we will not need it in this paper.

Let $(H,\eta)$ be a strip-structure of a graph $G$, and let $F \in E(H)$, where $\overline{F} = \{h_1,\ldots,h_k\}$.  Let $v_1,\ldots,v_k$ be new vertices, and let $J$ be the graph obtained from $G|\eta(F)$ by adding $v_1,\ldots,v_k$, where $v_i$ is complete to $\eta(F,h_i)$ and anticomplete to all other vertices of $J$.  Then $(J,\{v_1,\ldots,v_k\})$ is called the \emph{strip of} $(H,\eta)$ \emph{at} $F$.  A strip-structure $(H,\eta)$ is \emph{nontrivial} if $|E(H)| \geq 2$.
\\ \\
Appendix~\ref{appendix:strips} contains the descriptions of some strips $(J,Z)$ that we will need for the structure theorem.

We are now ready to state a structure theorem for claw-free graphs that 
is an easy corollary of the main result of \cite{claws5}.

\begin{theorem}\label{structure}
Let $G$ be a connected claw-free graph.  Then either
\begin{itemize}
\item $G$ is a member of $\mathcal{T}_1 \cup \mathcal{T}_2 \cup \mathcal{T}_3$, or
\item $V(G)$ is the union of three cliques, or
\item $G$ admits a nontrivial strip-structure such that for each strip $(J,Z)$, $1 \leq |Z| \leq 2$, and if $|Z|=2$, then either
\begin{itemize}
\item $|V(J)|=3$  and $Z$ is complete to $V(J) \setminus Z$, or
\item $(J,Z)$ is a member of $\mathcal{Z}_1  \cup \mathcal{Z}_2 \cup 
\mathcal{Z}_3 \cup \mathcal{Z}_4 \cup \mathcal{Z}_5$.
\end{itemize}
\end{itemize}
\end{theorem}

\section{Tools}
In this section we present a few lemmas that will then be used extensively in the following sections to prove results on the different graphs used in~\ref{structure}.

The following result is taken from \cite{Stiebitz}.  Because it is fundamental to many of our results, we include its proof here for completeness. 

\begin{lemma}\label{basic}
Let $G$ be a graph with chromatic number $\chi$ and let $K$ be a clique of size $k$ in $G$.  If $K$ is not Tihany, then every color class of a $(\chi-k)$-coloring of $G \setminus K$ contains a vertex complete to $K$.
\end{lemma}  

\begin{proof}
Suppose not. Since $K$ is not Tihany, it follows that $G \setminus K$ is $\chi-k$-colorable.  Let $C$ be a color class of a ($\chi-k$)-coloring of $G\setminus K$ with no vertex complete to $K$.  Define a coloring of $K \cup C$ by giving a distinct color to each vertex of $K$ and giving each vertex of $C$ a color of one of its non-neighbors in $K$.  This defines a $k$-coloring of $G|(K \cup C)$.  Note also that $G \setminus (K\cup C)$ is $\chi-k-1$-colorable.  However, this implies that $G$ is ($\chi-1$)-colorable, a contradiction.  This proves \ref{basic}.
\end{proof}

The next lemma is one of our most important and basic tool.  

\begin{lemma}\label{W-1}
Let $G$ be a graph such that \co. Let $K$ be a clique of $G$. If $K$ is dense, then it is Tihany.
\end{lemma}

\begin{proof}
Suppose that $K$ is not Tihany.  Let $\mathcal{C}$ be a $\chi-k$-coloring of $G\setminus K$.  By \ref{basic}, every color class of $\mathcal{C}$ contains a vertex complete to $K$.  Hence, every color class contains a member of $C(K)$ and so $|C(K) \cup K| \geq \chi(G)>\omega(G)$, a contradiction. This proves~\ref{W-1}.
\end{proof}

Let $(A,B)$ be disjoint subsets of $V(G)$.  The pair $(A,B)$ is called a \emph{homogeneous pair} in $G$ if $A,B$ are cliques, and for every vertex $v \in V(G) \setminus (A \cup B)$, $v$ is either $A$-complete or $A$-anticomplete and either $B$-complete or $B$-anticomplete.  A $W$-\emph{join} $(A,B)$ is a homogeneous pair in which $A$ is neither complete nor anticomplete to $B$.  We say that a $W$-join $(A,B)$ is \emph{reduced} if we can partition $A$ into two sets $A_1$ and $A_2$ and we can partition $B$ into $B_1,B_2$ such that $A_1$ is complete to $B_1$, $A_2$ is anticomplete to $B$, and $B_2$ is anticomplete to $A$.  Note that since $A$ is neither complete nor anticomplete to $B$, it follows that both $A_1$ and $B_1$ are non-empty and at least one of $A_2,B_2$ is non-empty.  We call a $W$-join that is not reduced a \emph{non-reduced} $W$-join.

Let $H$ be a thickening of $(G,F)$ and let $\{u,v\} \in F$.  Then we notice that $(X_u,X_v)$ is a $W$-join in $H$.  If for every $\{u,v\} \in F$ we have that $(X_u,X_v)$ is a reduced $W$-join then we say that $H$ is a \emph{reduced thickening} of $G$.  

The following result appears in \cite{twothirdschi}.

\begin{lemma}\label{homo}
Let $G$ be a claw-free graph and suppose that $G$ admits a non-reduced $W$-join.
Then there exists a subgraph $H$ of $G$  with the following properties:
\begin{enumerate}
\item $H$ is a claw-free graph,  $|V(H)|=|V(G)|$ and $|E(H)| < |E(G)|$.
\item $\chi(H) = \chi(G)$.
\end{enumerate}
\end{lemma}

The result of~\ref{homo} implies the following:
\begin{corollary}\label{reduced}
Let $G$ be a claw-free graph with \co~that is a minimal counterexample to~\ref{main}.
Assume also that $G$ is a thickening of  $(H,F)$ for some claw-free graph $H$ and $F \subseteq V(H)^2$. Then $G$ is a reduced 
thickening of $(H,F)$.
\end{corollary}

For a clique $K$ and $F\subseteq V(G)^2$, we define $S_F(K)=\{x:\exists k\in K \makebox{ s.t. } \{x,k\}\in F \makebox{ and } x\in C(K\backslash k)\}.$

\begin{lemma}\label{dense_thickening}
Let $G$ be a reduced thickening of $(H,F)$ for some claw-free graph $H$ and  $F \subseteq V(H)^2$.  
Let $K$ be a clique in $H$ such that for all $x,y\in C(K)$, $\{x,y\}\not\in F$. If $C(K)\cup S_F(K)$ is a clique, then there exists a dense clique of size $|K|$ in $G$.
\end{lemma}

\begin{proof}
Let $K'$ be a clique of size $|K|$ in $G$ such that $K'\cap X_v\neq \emptyset$ for all $v\in K$. By the definition of a thickening such a clique exists. Moreover since $C(K)\cup S_F(K)$ is a clique, it follows that $K'$ is dense. This proves~\ref{dense_thickening}.
\end{proof}
The following lemma is a direct corollary of~\ref{W-1} and~\ref{dense_thickening}.
\begin{lemma}\label{dense_thickening2}
Let $G$ be a reduced thickening of $(H,F)$ for some claw-free graph $H$ and  $F \subseteq V(H)^2$. Let $K$ be a dense clique in $H$ such that for all $x,y \in C(K)$, $\{x,y\} \not\in F$. If $C(K)\cup S_F(K)$ is a clique, then there exists a Tihany clique of size $|K|$ in $G$.
\end{lemma}

The following result helps us handle the case when $C(x)$ is an antimatching for some vertex $x \in V(G)$.

\begin{lemma}\label{lemma:vertexantimatching}
Let $G$ be a graph with \co. Let $u,x,y \in V(G)$ such that $ux,uy \in E(G)$ and $xy \not\in E(G)$.  Let $E=\{u,x\}$ and $E'=\{u,y\}$.  If $C(E)=C(E')$ then $E,E'$ are Tihany.
\end{lemma}  

\begin{proof}
Suppose that $E$ is not Tihany.  Let $\mathcal{C}$ be a $(\chi(G)-2)$-coloring of $G\setminus\{u,x\}$.  Let $C\in \mathcal{C}$ be the color class such that $y\in C$.  By Lemma~\ref{basic}, there is a vertex $z\in C$ such that $z$ is complete to $E$, and so $z \in C(E)$.  But $y$ is complete to $C(E)$, a contradiction.  Hence $E$ is be Tihany and by symmetry, so is $E'$.
\end{proof}

In particular, if we have a vertex $x$ such that $C(x)$ is an antimatching, we can find a Tihany edge either by~\ref{W-1} or by~\ref{lemma:vertexantimatching}. 

\begin{lemma}\label{lemma:goodclique}
Let $H$ be a graph, $G$ a thickening of $(G,F)$ for some valid
$F\subseteq G(V)^2$ such that \co. Let $K$ be a clique of $H$. Assume
that for all $\{x,y\}\in F$ such that $x\in K$, $y$ is complete to
$C(K)\backslash \{y\}$. Let $u,v\in C(K)$ such that $u$ is not
adjacent to $v$ and $\{u,v\}$ is complete to
$C(K)\backslash \{u,v\}$. Moreover assume that if there exists $E\in
F$ with $\{u,v\}\cap E\neq \emptyset$, then $E=\{u,v\}$. Then there
exists a Tihany clique of size $|K|+1$ in $G$.

\begin{proof}
Assume not. Let $K'$ be a clique of size $K$ in $G$ such that $K'\cap
X_y\neq \emptyset$ for all $y\in K$. If $\{u,v\}\notin F$, let $a\in
X_u$, $A=X_u$, $b\in X_v$ and $B=X_v$. If $\{u,v\}\in F$, let $X_u^1$,
$X_u^2$, $X_v^1$ and $X_v^2$ be as in the definition of reduced
W-join. By symmetry, we may assume that $X_u^2$ is not empty. If
$X_v^2$ is empty, let $a\in X_u^2$, $A=X_u^2$, $b\in X_v^1$ and
$B=X_v^1$; and if $X_v^2$ is not empty, let $a\in X_u^2$, $A=X_u$,
$b\in X_v^2$ and $B=X_v$.

Now let $T_a=K'\cup \{a\}$ and $T_b=K'\cup \{b\}$. We may assume that
$\chi(G\backslash T_a)=\chi(G\backslash
T_b)=\chi(G)-|K|-1$. By~\ref{basic}, we may assume that every color
class $G\backslash T_a$ contains a common neighbor of $T_a$. Since no
vertex of $B$ is complete to $T_1$, and since $B$ is a clique complete
to $C(T_1)\backslash A$, it follows that $|A| > |B|$. But similarly,
$|B|>|A|$, a contradiction. This proves~\ref{lemma:goodclique}.
\end{proof}
\end{lemma}

We need an additional definition before proving the next lemma. Let
$K$ be a clique; we denote by $\overline C(K)$ the closed neighborhood
of $K$, i.e. $\overline C(K):=C(K)\cup K$.

\begin{lemma}\label{lemma:disjoint-neighborhoods}
Let $G$ be a graph such that \co. Let $A$ and $B$ be cliques such that $2\leq |A|,|B| \leq 3$ (i.e., each one is a brace or a triangle). If $\overline C(A)\cap \overline C(B) = \emptyset$ and $\overline C(A)\cup \overline C(B)$ contains no triads then at least one of $A,B$ is Tihany.
\begin{proof}
Assume not and let $k=\chi(G)-|A|$. By~\ref{basic}, in every $k$-coloring of $G\backslash A$ every color class must have a vertex in $C(A)$. As there is no triad in $\overline C(A)\cup \overline C(B)$, it follows that every vertex of $C(A)$ is in a color class with at most one vertex of $\overline C(B)$, thus $\overline C(A)> \overline C(B)$. By symmetry, it follows that $\overline C(A) < \overline C(B)$, a contradiction. This proves~\ref{lemma:disjoint-neighborhoods}.
\end{proof}
\end{lemma}

\begin{lemma}\label{cliquecutset}
Let $G$ be a claw-free graph such that \co. If $G$ admits a clique cutset, then there is a Tihany brace in $G$.
\begin{proof}
Let $K$ be a clique cutset. Let $A,B\subset V(G)\backslash K$ such that $A\cap B=\emptyset$ and $A\cup B\cup K=V(G)$. Let $\chi_A= \chi(G|(A\cup K))$ and $\chi_B = \chi(G|(B\cup K))$. By symmetry, we may assume that $\chi_A\geq \chi_B$.
\begin{claim}\label{claimcliquecutset}
$\chi(G)=\chi_A$
\end{claim}
Let $\mathcal{S}_{A}=(A_1,A_2,\ldots,A_{\chi_{A}})$ and $\mathcal{S}_B=(B_1,B_2,\ldots,B_{\chi_{B}})$ be optimal coloring of $G|(A\cup K)$ and $G|(B\cup K)$. Let $K=\{k_1,k_2,\ldots,K_{|K|}\}$. Up to renaming the stable sets, we may assume that $A_i\cap B_i=\{k_i\}$ for all $i=1,2,\ldots,|K|$. Then $\mathcal{S}=(A_1\cup B_1,A_2\cup B_2,\ldots,A_{\chi_B}\cup B_{\chi_B},A_{\chi_B+1},\ldots ,A_{\chi_A}\}$ is a $\chi_A$-coloring of $G$. This proves~\cref{claimcliquecutset}.\\

Now let $x\in B$ and $y\in K$ be such that $xy\in E(G)$. Then $\chi(G\backslash \{x,y\})\geq \chi(G|(A\cup K\backslash \{y\})\geq \chi_A-1 \geq \chi(G)-1$. Hence $\{x,y\}$ is a Tihany brace. This proves~\ref{cliquecutset}. 
\end{proof}
\end{lemma}

\section{The Icosahedron and Long Circular Interval Graphs}
\begin{theorem}\label{icosa}
Let $G\in \mathcal{T}_1$. If \co, then there exists a Tihany brace in $G$.
\end{theorem}

\begin{proof}
Let $v_0,v_1,\ldots,v_{11}$ be as in the definition of the
icosahedron. Let $G_0,G_1,G_2$, and $F$ be as in the definition of
$\mathcal{T}_1$.  Then $G$ is a thickening of either
$(G_0,\emptyset)$, $(G_1,\emptyset)$, or $(G_2,F)$ for $F \subseteq
\{(v_1,v_4),(v_6,v_9)\}$.  For $0 \leq i \leq 11$, let $X_{v_i}$ be as
in the definition of thickening (where $X_{v_{11}}$ is empty when $G$
is a thickening of $(G_1,\emptyset)$ or $(G_2,F)$, and $X_{v_{10}}$ is
empty when $G$ is a thickening of $(G_2,F)$).  Let $x_i \in X_{v_i}$
and $w_i=|X_{v_i}|$.

First suppose that $G$ is a thickening of $(G_1,\emptyset)$ or
$(G_2,F)$. Then $C(\{x_4,x_6\})=X_{v_4} \cup X_{v_5} \cup X_{v_6}$ is
a clique. Therefore, $\{x_4,x_6\}$ is a Tihany brace by~\ref{W-1}.

So we may assume that $G$ is a thickening of
$(G_0,\emptyset)$. Suppose that no brace of $G$ is Tihany and let
$E=\{x_1,x_3\}$.  Then $G\backslash E$ is $(\chi-2)$-colorable.
By~\ref{basic}, every color class contains at least one vertex from
$C(E)=(X_1\cup X_2\cup X_3\cup X_0)\setminus\{x_1,x_3\}$. Since
$\alpha(G)=3$, it follows that every color class has at most two
vertices from $\bigcup_{i=4}^{11} X_{v_{i}}$. Hence we conclude that
$$w_4+w_5+w_6+w_7+w_8+w_9+w_{10}+w_{11} \leq 2\cdot(w_1+w_2+w_3+w_0
-2)$$ A similar inequality exists for every brace
$\{x_i,x_j\}$. Summing these inequalities over all braces
$\{x_i,x_j\}$, it follows that $(\sum_{i=0}^{11} 20w_i) \leq
(\sum_{i=0}^{11} 20w_i) - 120$, a contradiction. This
proves~\ref{icosa}.
\end{proof}

\begin{theorem}\label{longcircularinterval}
Let $G\in \mathcal{T}_2$. If \co, then there exists a Tihany brace in $G$.
\begin{proof}
Let $H,F,\Sigma,F_1,\ldots,F_k$ be as in the definition of
$\mathcal{T}_2$ such that $G$ is a thickening of $(H,F)$. Let $F_i$ be
such that there exists no $j$ with $F_i\subset F_j$. Let
$\{x_k,\ldots,x_l\}=V(H)\cap F_i$ and without loss of generality, we
may assume that $\{x_k,\ldots,x_l\}$ are in order on $\Sigma$. Since
$C(\{x_k,x_l\})=\{x_{k+1},\ldots,x_{l-1}\}$, it follows that
$\{x_k,x_l\}$ is dense. Hence by~\ref{dense_thickening2} there exists
a Tihany brace in $G$. This proves~\ref{longcircularinterval}.
\end{proof}
\end{theorem}

\section{Non-2-substantial and Non-3-substantial Graphs}
In this section we study graphs where a few vertices cover all the triads. An antiprismatic graph $G$ is \textit{k-substantial} if for every $S\subseteq V(G)$ with $|S|<k$ there is a triad $T$ with $S \cap T = \emptyset$. The \textit{matching number} of a graph $G$, denoted by $\mu(G)$, is the number of edges in a maximum matching in $G$. Balogh et al.~\cite{quasi-line} proved the following theorem.

\begin{theorem}\label{alpha2}
Let $G$ be a graph such that $\alpha(G)=2$ and $\chi(G)>\omega(G)$. For any two integers $s,t\geq 2$ such that $s+t=\chi(G)+1$ there exists a partition $(S,T)$ of $V(G)$ such that $\chi(G|S)\geq s$ and $\chi(G|T)\geq t$.
\end{theorem}

The following theorem is a result of Gallai and Edmonds on matchings and it will be used in the study of non-2-substantial and non-3-substantial graphs.

\begin{theorem}[Gallai-Edmonds Structure Theorem \cite{edmonds}, \cite{gallai}]\label{gallaiedmonds}
Let $G = (V,E)$ be a graph. Let $D$ denote the set of nodes which are not covered by at
least one maximum matching of $G$. Let $A$ be the set of nodes in $V\backslash D$ adjacent to at
least one node in $D$. Let $C = V\backslash (A \cup D)$. Then:
\begin{enumerate}
\item[i)] The number of covered nodes by a maximum matching in G equals to $|V|+|A|-c(D)$, where $c(D)$ denotes the number of components of the graph spanned by $D$.
\item[ii)] If $M$ is a maximum matching of $G$, then for every component $F$ of $D$, $E(D)\cap M$ covers all but one of the nodes of $F$, $E(C) \cap M$ is a perfect matching and M matches
all the nodes of $A$ with nodes in distinct components of $D$.
\end{enumerate}
\end{theorem}

\begin{lemma}\label{neighborcobipartite}
Let $G$ be an antiprismatic graph. Let $K$ be a clique and $u,v\in V(G)\backslash \overline C(K)$ be non-adjacent. If $\alpha(G|(C(K)\cup\{u,v\}))= 2$ and $\alpha(G|K\cup \{u,v\})=3$, then $G|\overline C(K)$ is cobipartite. 
\begin{proof}
Since there is no triad in $C(K)\cup \{u,v\}$, we deduce that there is no vertex in $C(K)$ anticomplete to $\{u,v\}$. Since $G$ is claw-free and $\alpha(G|K\cup \{u,v\})=3$, it follows that there is no vertex in $C(K)$ complete to $\{u,v\}$. Let $C_u,C_v\subseteq C(K)$ be such that $C_u\cup C_v = C(K)$ and for all $x\in C(K)$, $x$ is adjacent to $u$ and non-adjacent to $v$ if $x\in C_u$, and $x$ is adjacent to $v$ and non-adjacent to $u$ if $x\in C_v$. Since $\alpha(G|(C_v\cup\{u\}))=2$, we deduce that $C_v$ is a clique and by symmetry $C_u$ is a clique. Hence $\overline C(K)$ is the union of two cliques. This proves~\ref{neighborcobipartite}.
\end{proof}
\end{lemma}
\begin{theorem}\label{nonsubstantialsmallclique}
Let $G$ be a claw-free graph such that \co. Let $K$ be a clique such that $\alpha(G\backslash K)\leq 2$. Then there exists a Tihany clique of size at most $|K|+1$ in $G$.

\begin{proof}
Assume not. Let $n=|V(G)|$, $w\in C(K)$ and $K'=K\cup \{w\}$ (such a vertex $w$ exists since $K$ is not Tihany).
\begin{claim}\label{claim:nssc1}
$\chi(G)=n-\mu(G^c)$.
\end{claim}
Since $K'$ is not Tihany, it follows that $\chi(G\backslash K')=\chi(G)-|K'|$. Since $\alpha(G\backslash K')\leq 2$, we deduce that $\chi(G\backslash K') \geq \frac{n-|K'|}{2}$, and thus $\chi(G) \geq \frac{n+|K'|}{2}$. Hence in every optimal coloring of $G$ the color classes have an average size strictly smaller than $2$, and since $G$ is claw-free, we deduce that there is an optimal coloring of $G$ where all color classes have size $1$ or $2$. It follows that $\chi(G)\leq n-\mu(G^c)$. But clearly $\chi(G) \geq n-\mu(G^c)$, thus $\chi(G)=n-\mu(G^c)$. This proves~\cref{claim:nssc1}. 
\begin{claim}\label{claim:nssc2}
Let $T$ be a clique of size $|K|+1$ in $G$, then $\chi(G\backslash T)=n-|T|-\mu(G^c\backslash T)$.
\end{claim}
Since $T$ is not Tihany, it follows that $\chi(G\backslash T) = \chi(G)-|T| \geq \frac{n+|K'|}{2}-|T|=\frac{n-|T|}{2} = \frac{|V(G\backslash T)|}{2}$. Hence in every optimal coloring of $G\backslash T$, the color classes have an average size smaller than $2$, and since $G$ is claw-free, we deduce that there is an optimal coloring of $G\backslash T$ where all color classes have size $1$ or $2$. It follows that $\chi(G\backslash T)\leq |V(G\backslash T|-\mu(G^c\backslash T)$. Hence $\chi(G\backslash T) = n-|T|-\mu(G^c\backslash T)$. This proves~\cref{claim:nssc2}.\\

Let $A,D,C$ be as in~\ref{gallaiedmonds}. Since $\chi(G)\geq \frac{n+|K'|}{2}$ and $\chi(G)=n-\mu(G^c)$, we deduce that $\mu(G^c)\leq \frac{n-|K'|}{2}$. By~\ref{gallaiedmonds} i), we deduce that $\mu(G^c) =\frac{n+|A|-c(D)}{2}$. Thus, it follows that $c(D)\geq |K'|$. Let $D_1,D_2,\ldots,D_{c(D)}$ be the anticomponents of $D$. Let $d_i\in D_i$ for $i=1,\ldots, c(D)$.

\begin{claim}\label{claim:nssc3}
$|D_i|=1$ for all $i$. 
\end{claim}
Assume not and by symmetry assume that $|D_1|>1$. Since $G$ is claw-free, we deduce that $\alpha(G|D_1)=2$. Thus there exist $x,y\in D_1$ such that $x$ is adjacent to $y$. Now $T=\{x,y,d_2,\ldots,d_{|K|}\}$ is a clique of size $|K|+1$ and by~\ref{gallaiedmonds} ii), it follows that $\mu(G^c\backslash T)<\mu(G^c)$. By~\cref{claim:nssc1} and~\cref{claim:nssc2}, it follows that $\chi(G\backslash T)+|T|=n-\mu(G^c\backslash T)> n-\mu(G^c) = \chi(G)$, a contradiction. This proves~\cref{claim:nssc3}.\\

Let $T=\{d_1,\ldots,d_{|K|+1}\}$. By~\cref{claim:nssc3}, it follows that $C(T)\cap D$ is a clique. By~\ref{W-1}, we deduce that $C(T)\cap A\neq \emptyset$. Let $x\in C(T)\cap A$. Now $S=\{d_1,\ldots, d_{|K|},x\}$ is a clique of size $|K|+1$ and by~\ref{gallaiedmonds} ii), it follows that $\mu(G^c\backslash S)<\mu(G^c)$. By~\cref{claim:nssc1} and~\cref{claim:nssc2}, it follows that $\chi(G\backslash S)+|S|=n-\mu(G^c\backslash S)> n-\mu(G^c) = \chi(G)$, a contradiction. This concludes the proof of~\ref{nonsubstantialsmallclique}.
\end{proof}

\end{theorem}

\begin{theorem}\label{heavynon2sub}
Let $H$ be a claw-free graph such that there exists $x\in V(H)$ with $\alpha(H\backslash x)=2$. Let $G$ be a reduced thickening of $H$ such that \co~and $|X_x|>1$. Then for all $\{u,v\}\in X_x$, \cGuv.

\begin{proof}
Let $u,v\in X_x$. We may assume that $\{u,v\}$ is not Tihany. Let $k=\chi(G\backslash\{u,v\})$ and $\mathcal{S}=(S_1,S_2,\ldots, S_k)$ be a $k$-coloring of $G\backslash \{u,v\}$. By~\ref{basic}, $S_i\cap C(\{u,v\})\neq \emptyset$. Let $I_l=\{i:|S_i|=l\}$ and let $O=C(\{u,v\})\cap \bigcup_{i\in I_1\cup I_2} S_i$ and $P=C(\{u,v\})\cap \bigcup_{i\in I_3} S_i$. 

Since $\alpha(H\backslash x)=2$, it follows that $S_i \cap X_x\neq \emptyset $ for all $i\in I_3$. Hence, $P$ is a clique complete to $O$ and thus $\omega(G|O\cup P) = \omega(G|O) + |I_3|$. Since \co, we deduce that $\omega(G|O)<|I_1\cup I_2|$. By~\ref{neighborcobipartite} and since $O\subseteq \overline C(X_x)$, we deduce that $G|O$ is cobipartite. Hence $\chi(G|O)=\omega(G|O)<|I_1\cup I_2|$. Thus the coloring $\mathcal{S}$ does not induce an optimal coloring of $G|O$. It follows that there exists an augmenting antipath $P=p_1-p_2-\ldots-p_{2l}$ in $O$. Now let $T_i=\{p_{2i-1},p_{2i}\}$ for $i=1,\ldots,l$. Let $s$ be such that $p_1\in S_s$ and $e$ be such that $p_{2l}\in S_e$. They are the color classes where the augmenting antipath starts and ends. If $|S_s|=2$, let $T_{l+1}=(\{u\}\cup S_s\backslash p_1)$, otherwise let $T_{l+1}=\{u\}$. If $|S_e|=2$, let $T_{l+2}=(\{v\}\cup S_e\backslash p_{2l})$, otherwise let $T_{l+2}=\{v\}$. Let $J=\{i| S_i\cap V(P) \neq \emptyset\}$. Clearly $|J|= l+1$. Now $(T_1,T_2,\ldots, T_{l+2})$ is a (l+2)-coloring of $\bigcup_{i\in J} S_i \cup \{u,v\}$, which together with the color classes $S_i$ for $i\notin J$ create a $k+1$-coloring of $G$, a contradiction. This proves~\ref{heavynon2sub}.
\end{proof}
\end{theorem}
The next lemma is a direct corollary of~\ref{nonsubstantialsmallclique} and~\ref{heavynon2sub}.
\begin{lemma}\label{non2sub}
Let $H$ be a non-2-substantial claw-free graph. Let $G$ be a reduced thickening of an augmentation of $H$ such that \co. Then there exists a Tihany brace in $G$.
\end{lemma}

Now we look at non-3-substantial graphs.

\begin{lemma}\label{theorem:not3substantial}
Let $H$ be a non-3-substantial antiprismatic graph. Let $u,v\in H$ be such that $\alpha(H\backslash \{u,v\})=2$.
Let $G$ be a reduced thickening of $H$ such that \co. If $u$ is not adjacent to $v$, then there exists a Tihany brace or triangle in $G$.
\begin{proof}
Assume not. Let $N_u=C(u)\backslash C(\{u,v\}$ and $N_v=C(u)\backslash C(\{u,v\})$. Since $H$ is antiprismatic, it follows that $N_u$ and $N_v$ are antimatchings.

By~\ref{non2sub}, we deduce that $N_u$ and $N_v$ are not cliques. Let $x_u,y_u\in N_u$ be not adjacent, and $x_v,y_v\in N_v$ be not adjacent. Since $\alpha(H\backslash \{u,v\})=2$ and $H$ is antiprismatic, we may assume by symmetry that $x_ux_v,y_uy_v$ are edges, and $x_uy_v,y_ux_v$ are non-edges.
Since $\alpha(h\backslash \{u,v\})=2$ and $H$ is antiprismatic, it follows that every vertex in $C(\{u,v\})$ is either strongly complete to $x_ux_v$ and strongly anticomplete to $y_uy_v$, or strongly complete to $y_uy_v$ and strongly anticomplete to $x_ux_v$. Let $(N_x,N_y)$ be the partition of $C(\{u,v\})$ such that all $x\in N_x$ are complete to $x_ux_v$ and and all $y\in N_y$ are complete to $y_uy_v$.

Assume first that $N_x\neq \emptyset$ and $N_y\neq \emptyset$. Let $n_x\in N_x$ and $n_y\in N_y$ and let $T_u=\{u,y_u,n_y\}$ and $T_v=\{v,x_v,n_x\}$. Clearly $T_u$ and $T_v$ are triangles. 
\begin{claim}\label{claim:not3sub}
$\alpha(G|(\overline C(T_u)\cup \overline C(T_v))=2$ and $\overline C(T_u)\cap \overline C(T_v)=\emptyset$.
\end{claim}
Assume not. Since $\overline C(T_u)\subseteq N_y\cup N_u\cup \{u\}$ and $\overline C(T_v)\subseteq N_x\cup N_v\cup \{v\}$, we deduce that $\overline C(T_u)\cap \overline C(T_v)=\emptyset$. Let $T\in \overline C(T_u)\cup \overline C(T_v)$ be a triad. By symmetry, we may assume that $u\in T$. Clearly, $T\backslash u\in N_v$. But since $H$ is antiprismatic, we deduce that $T\backslash u\subseteq C(n_x)$, hence $T\backslash u\notin \overline C(T_u)\cup \overline C(T_v)$, a contradiction. This proves~\cref{claim:not3sub}.\\

Now let $S_u,S_v\in G$ be triangles such that $|S_u\cap X_u|=|S_u\cap X_{y_u}|=|S_u\cap X_{n_y}|=1$ and $|S_v\cap X_v|=|S_v\cap X_{x_v}|=|S_v\cap X_{n_x}|=1$. By~\cref{claim:not3sub} and~\ref{lemma:disjoint-neighborhoods} and since $G$ is a reduced thickening of $H$, we deduce that there is a Tihany triangle in $G$.

Now assume that at least one of $N_x,N_y$ is empty. By symmetry, we may assume that $N_x$ is empty. Since $C(\{u,x_u\})$ is an antimatching, by~\ref{lemma:goodclique} there exists a Tihany triangle in $G$. This concludes the proof of~\ref{theorem:not3substantial}.
\end{proof}
\end{lemma}

\begin{lemma}\label{theorem:not3substantialadj}
Let $H$ be a non-3-substantial antiprismatic graph. Let $u,v\in H$ be such that $\alpha(G\backslash \{u,v\})=2$.
Let $G$ be a reduced thickening of $(H,F)$ for some valid $F\subseteq V(G)^2$ such that \co. If $u$ is adjacent to $v$, then there exists a Tihany clique $K$ in $G$ with $|K|\leq 4$.
\begin{proof}

Assume not. By~\ref{nonsubstantialsmallclique}, we may assume that $|X_u\cup X_v|>2$. By~\ref{non2sub}, we may assume that $|X_u|>0$ or $|X_v|>0$. If $|X_u|=1$, then $G\backslash X_u$ is a reduced thickening of a non-2-substantial antiprismatic graph. By~\ref{heavynon2sub}, there exists a brace $\{x,y\}$ in $X_v$ such that $\chi(G\backslash(\{x,y\}\cup X_u)) \geq \chi(G\backslash X_u)-1$. But $\chi(G\backslash X_u)-1 \geq \chi(G)-2$, hence $\{x,y\}\cup X_u$ is a Tihany triangle, a contradiction. Thus $|X_u|>1$, and by symmetry $|X_v|>1$.

Let $x_1,y_1\in X_u$ and $x_2,y_2\in X_v$, thus $C=\{x_1,x_2,y_1,y_2\}$ is a clique of size $4$. 

Let $k=\chi(G\backslash C)$ and $\mathcal{S}=(S_1,S_2,\ldots, S_k)$ be a $k$-coloring of $G\backslash C$. By~\ref{basic}, $S_i\cap N(C)\neq\emptyset$. For $l=1,2,3$ let $I_l=\{i:|S_i|=l\}$ and let $O=N(C)\cap \bigcup_{i\in I_1\cup I_2} S_i$ and $P=N(C)\cap \bigcup_{i\in I_3} S_i$. 

Since $\alpha(H\backslash \{u,v\})=2$, it follows that $S_i \cap (X_u\cup X_v)\neq \emptyset $ for all $i\in I_3$. Hence, $\omega(G|O\cup P) = \omega(G|O) + |I_3|$. Since \co, we deduce that $\omega(G|O)<|I_1\cup I_2|$. By~\ref{neighborcobipartite}, we deduce that $G|O$ is cobipartite. Hence $\chi(G|O)=\omega(G|O)<|I_1|+|I_2|$. Thus the coloring $\mathcal{S}$ does not induce an optimal coloring of $G|O$. It follows that there exists an augmenting antipath $P=p_1-p_2-\ldots-p_{2l}$ in $O$. Now let $T_i=\{p_{2i-1},p_{2i}\}$ for $i=1,\ldots,l$. Let $s$ be such that $p_1\in S_s$ and $e$ be such that $p_{2l}\in S_e$. They are the color classes where the augmenting antipath starts and ends. 
Since $S_s\backslash p_1$ is not complete to $\{x_1,y_1\}$, we deduce that there exists $\hat s\in \{1,2\}$ such that $x_{\hat s}$ is antiadjacent to $S_s\backslash p_1$. Let $T_{l+1}=\{x_{\hat s}\}\cup S_s \backslash p_1$ and $T_{l+2}=\{x_1,x_2\}\backslash x_{\hat s}$. 
Since $S_e\backslash p_{2l}$ is not complete to $\{x_2,y_2\}$, we deduce that there exists $\hat e\in \{1,2\}$ such that $x_{\hat e}$ is antiadjacent to $S_e\backslash p_{2l}$. Let $T_{l+3}=\{x_{\hat e}\}\cup S_e\backslash p_{2l}$ and $T_{l+4}=\{y_1,y_2\}\backslash x_{\hat e}$.

Let $J=\{i| S_i\cap V(P) \neq \emptyset\}$. Clearly $|J|= l+1$. Now $(T_1,T_2,\ldots, T_{l+2},T_{l+3},T_{l+4})$ is a (l+4)-coloring of $\bigcup_{i\in J} S_i \cup \{x_1,x_2,y_1,y_2\}$, which together with the color classes $S_i$, for $i\notin J$, create a $k+3$-coloring of $G$, a contradiction. This proves~\ref{theorem:not3substantialadj}.
\end{proof}
\end{lemma}

The following lemma is a direct corollary of~\ref{theorem:not3substantial} and~\ref{theorem:not3substantialadj}.

\begin{lemma} \label{mainnot3sub}
Let $H$ be a non-3-substantial antiprismatic graph. Let $G$ be a reduced thickening of $H$ such that \co. Then there exists a Tihany clique $K \subset V(G)$ with $|K|\leq 4$.
\end{lemma}

\section{Complements of orientable prismatic graphs}
\label{sec:orientable} In this section we study the complements of orientable prismatic graphs. A graph is \textit{prismatic} if its complement is antiprismatic. Let $G$ be a graph. The \textit{core} of $G$ is the union of all the triangles in $G$. If $\{a,b,c\}$ is a triangle in $G$ and both $b,c$ only belong to one triangle in $G$, then $b$ and $c$ are said to be \textit{weak}. The \textit{strong core} of $G$ is the subset of the core such that no vertex in the strong core is weak. As proved in~\cite{claws1}, if $H$ is a thickening of $(G,F)$ for some valid $F\subseteq V(G)^2$ and $\{x,y\}\in F$, then $x$ and $y$ are not in the strong core.

A {\em three-cliqued claw-free graph} $(G,A,B,C)$ consists of a claw-free graph $G$ and three cliques $A,B,C$ of $G$, pairwise disjoint and with union $V(G)$. The complement of a tree-cliqued graph is a \textit{3-coloured graph}.
Let $n\ge 0$, and for $1\le i \le n$, let $(G_i,A_i,B_i,C_i)$ be a three-cliqued graph,
where $V(G_1)\l V(G_n)$ are all nonempty and pairwise vertex-disjoint.
Let $A = A_1\cup \cdots \cup A_n$,
$B = B_1\cup \cdots \cup B_n$, and $C = C_1\cup \cdots \cup C_n$, and let $G$ be the graph with vertex set
$V(G_1)\cup \cdots \cup V(G_n)$ and with adjacency as follows:
\begin{itemize}
\item for $1\le i \le n$, $G|V(G_i) = G_i$;
\item for $1\le i<j\le n$, $A_i$ is complete to $V(G_j)\setminus B_j$; $B_i$ is complete to $V(G_j)\setminus C_j$;
and $C_i$ is complete to $V(G_j)\setminus A_j$; and
\item for $1\le i<j\le n$, if $u\in A_i$ and $v\in B_j$ are adjacent then $u,v$ are both in no triads; and the same
applies if $u\in B_i$ and $v\in C_j$, and if $u\in C_i$ and $v\in A_j$.
\end{itemize}
In particular, $A,B,C$ are cliques, and so $(G,A,B,C)$ is a three-cliqued graph and $(G^c,A,B,C)$ is a 3-coloured graph; we call the sequence $(G_i,A_i,B_i,C_i)$ $(i = 1\l n)$ a {\em worn hex-chain} for $(G,A,B,C)$. When $n = 2$ we say that $(G,A,B,C)$ is a {\em worn hex-join} of $(G_1,A_1,B_1,C_1)$ and $(G_2,A_2,B_2,C_2)$. Similarly, the sequence $(G^c_i,A_i,B_i,C_i)\;(i = 1\l n)$ is a {\em worn hex-chain} for $(G^c,A,B,C)$, and when $n = 2$, $(G^c,A,B,C)$ is a {\em worn hex-join} of $(G^c_1,A_1,B_1,C_1)$ and $(G^c_2,A_2,B_2,C_2)$.
Note also that every triad of $G$ is a triad of one of $G_1\l G_n$. If each $G_i$ is claw-free then so is $G$ and if each $G^c_i$ is prismatic then so is $G^c$.

If $(G,A,B,C)$ is a three-cliqued graph, and $\{A',B',C'\}  = \{A,B,C\}$, then $(G,A',B',C')$ is also a three-cliqued graph, that we say is a {\em permutation} of $(G,A,B,C)$.

A list of the definitions needed for the study of the prismatic graphs can be found in appendix~\ref{app:orientable}. The structure of prismatic graph has been extensively studied in \cite{claws1} and \cite{claws2}; the resulting two main theorems are the following.
\begin{lemma}\label{decomp_orientable_prism}
Every orientable prismatic graph that is not 3-colourable is either not 3-substantial, or a cycle of triangles graph, or a ring of five graph, or a mantled $L(K_{3,3})$.
\end{lemma}

\begin{lemma}\label{decomp_3cliques}
Every 3-coloured prismatic graph admits a worn chain decomposition with all terms in $\mathcal Q_0\cup\mathcal Q_1\cup \mathcal Q_2$.
\end{lemma}

In the remainder of the section, we use these two results to prove our main theorem for complements of orientable prismatic graphs.  We begins with some results that deal with the various outcomes of \ref{decomp_orientable_prism}.
\begin{lemma}\label{cycleoftriangles}
Let $H$ be a prismatic cycle of triangles and $G$ be a reduced thickening of $(\overline H,F)$ for some valid $F\in V(G)^2$ such that $\chi(G)>\omega(G)$. Then there exists a Tihany brace or triangle in $G$.
\begin{proof}
Let the set $X_i$ be as in the definition of a cycle of triangles. Up to renaming the sets, we may assume $|\hat X_{2n}|=|\hat X_4|=1$. Let $u\in \hat X_{2i}$ and $v\in \hat X_4$; hence $uv$ is an edge. We have 
$$C_H(\{u,v\})=\bigcup_{j=1 \mod 3, j\geq 4}X_{j} \cup R_1 \cup L_3.$$

If $|\hat X_2|>1$, then $|R_1|=|L_3|=\emptyset$ and so $C_H(\{u,v\})$ is a clique. Therefore by~\ref{dense_thickening2}, there is a Tihany brace in $G$. If $|\hat X_2|=1$, the only non-edges in $\overline G|C_H(\{u,v\})$ are a perfect anti-matching between $R_1$ and $L_3$. Hence by~\ref{lemma:goodclique}, there is a Tihany triangle in $G$. This proves~\ref{cycleoftriangles}
\end{proof}
\end{lemma}

\begin{lemma}\label{ringoffive}
Let $H$ be a ring of five graph. Let $G$ be a reduced thickening of $(\overline H,F)$ for some valid $F\in V(G)^2$ such that \co. Then there is a Tihany triangle in $G$.
\begin{proof}
Let $a_2,b_3,a_4$ be as in the definition of a ring of five. $C(\{a_2,b_3,a_4\})=V_2\cup V_4$ and thus $\{a_2,b_3,a_4\}$ is a dense triangle. By the definitions of $H$ and $F$, it follows that $\{a_2,b_3,a_4\}\cap E=\emptyset$ for all $E\in F$. Hence by~\ref{dense_thickening2}, there exists a Tihany triangle in $G$. This proves~\ref{ringoffive}.
\end{proof}
\end{lemma}

\begin{lemma}\label{mantled}
Let $H$ be a mantled $L(K_{3,3})$ and $G$ be a reduced thickening of $(\overline H,F)$ for some valid $F\in V(G)^2$. If \co, then there exists a Tihany brace in $G$.
\begin{proof}
Let $W,a_j^i,V^i,V_i$ be as in the definition of mantled $L(K_{3,3})$. Let $X_j^i$ be the clique corresponding to $a_j^i$ in the thickening and $\mathcal{W}$ (resp. $\mathcal{V}_i$, $\mathcal{V}^i$) be the set of vertices corresponding to $W$ (resp. $V_i$, $V^i$) in the thickening. Let $x_i^j \in X_i^j$, $\mathcal{V}=\cup_{i=1}^3 \mathcal{V}_i\cup\mathcal{V}^i$  and $k=\chi(G)$.

For a brace $E$ in $G$, let $M_W(E):=M(E)\cap \mathcal{W}$, $M_V(E):=M(E)\cap \mathcal{V}$, $A_W(E):=A(E)\cap \mathcal{W}$ and $A_V(E):=A(E)\cap \mathcal{V}$. Let $E=\{x_i^j,x_{i'}^{j'}\}$ and let $S$ be a color class in a (k-2)-colouring of $G\backslash E$.
\begin{claim}\label{claim1mantel}
If $S\cap A_V(E)\neq \emptyset$, then $|S|\leq 2$. 
\end{claim}
Assume not. Let $S=\{x,y,z\}$ and without loss of generality we may assume that $E=\{x_1^1,x_2^1\}$ and $x\in A_V(E)=\mathcal{V}^1$.  Since $x$ is complete to $\mathcal{V}_1\cup \mathcal{V}_2 \cup \mathcal{V}_3$ and $X_i^j$, for $i=1,2,3\ j=2,3$, we deduce that $y,z\notin \mathcal{V}_1\cup \mathcal{V}_2 \cup \mathcal{V}_3$ and $y,z\notin X_i^j$, for $i=1,2,3\ j=2,3$. Since there is no triad in $\mathcal{V}^1\cup \mathcal{V}^2 \cup\mathcal{ V}^3$, it follows that $|\{y,z\}\cap ( \mathcal{V}^1\cup \mathcal{V}^2 \cup\mathcal{ V}^3)|\leq 1$. Since $X_1^1\cup X_2^1\cup X_3^1$ is a clique, we deduce that $|\{y,z\}\cap (X_1^1\cup X_2^1\cup X_3^1)|\leq 1$. Hence, we may assume by symmetry that $y\in X_1^1\cup X_2^1\cup X_3^1$ and $z\in \mathcal{V}^2\cup \mathcal{V}^3$. But $X_1^1\cup X_2^1\cup X_3^1$ is complete to $\mathcal{V}^2\cup \mathcal{V}^3$, a contradiction. This proves~\cref{claim1mantel}.

\begin{claim}\label{claim2mantel}
If $S\cap M_V(E)\neq \emptyset$, then $|S|\leq 2$. 
\end{claim}
Assume not. Let $S=\{x,y,z\}$ and without loss of generality we may assume that $E=\{x_1^1,x_2^1\}$ and $x\in \mathcal{V}_1$. Since $x$ is complete to $\mathcal{V}^1\cup \mathcal{V}^2 \cup \mathcal{V}^3$ and $X_2^j\cup X_3^j$, for $j=1,2,3$, we deduce that $y,z\notin \mathcal{V}^1\cup \mathcal{V}^2 \cup \mathcal{V}^3$ and $y,z\notin X_2^j\cup X_3^j$, for $j=1,2,3$. Since there is no triad in $\mathcal{V}_2\cup \mathcal{V}_3$, it follows that $|\{y,z\}\cap \mathcal{V}_2\cup \mathcal{V}_3|\leq 1$. As $X_1^1\cup X_1^2\cup X_1^3$ is a clique, we deduce that $|\{y,z\}\cap (X_1^1\cup X_1^2\cup X_1^3)|\leq 1$. Hence we may assume by symmetry that $y\in \mathcal{V}_2\cup \mathcal{V}_3$ and $z\in X_1^1\cup X_1^2\cup X_1^3$. But $\mathcal{V}_2\cup \mathcal{V}_3$ is complete to $X_1^1\cup X_1^2\cup X_1^3$, a contradiction. This proves~\cref{claim2mantel}.\\

By~\ref{basic}, every color class of a $(k-2)$-coloring of $G\backslash E$ must have a vertex in $C(E)$. By~\cref{claim1mantel} and~\cref{claim2mantel}, it follows that color classes with vertices in $A_V(E)\cup M_V(E)$ have size 2. Hence we deduce that $A_V(E)+M_V(E)+\frac{1}{2}A_W(E)+\frac{1}{2}M_W(E) \leq C(E)-2$.
Summing this inequality on all baces $E=\{x_i^j,x_{i'}^{j'}\}\ i,j=1,2,3$, it follows that
$$3\sum_i(|\mathcal{V}_i|+|\mathcal{V}^i|)+ 6\sum_i(|\mathcal{V}_i|+|\mathcal{V}^i|) + \frac{4}{2}\sum_{i,j} |X_i^j| + \frac{8}{2}\sum_{i,j} |X_i^j|<9\sum_i(|\mathcal{V}_i|+|\mathcal{V}^i|) + 6\sum_{i,j} |X_i^j|,$$
which is a contradiction. This proves~\ref{mantled}.
\end{proof}
\end{lemma}
\begin{lemma}\label{path_of_triangle}
Let $(H,H_1,H_2,H_3)^c$ be a path of triangle and $(I,I_1,I_2,I_3)$ an antiprismatic three-cliqued graph. Let $G$ be a worn hex-join of $(H,H_1,H_2,H_3)$ and $(I,I_1,I_2,I_3)$, and $G'$ be a reduced thickening of $(G,F)$ for some valid $F\in V(G)^2$ such that $\chi(G')>\omega(G')$. Then there exists a Tihany clique $K$ in $G'$, with $|K|\leq 4$.
\begin{proof}
Assume not. Let the set $X_j$ of $H$  be as in the definition of a path of triangle and we may assume that $H_i=\cup_{j=i\mod 3}X_j$.

Assume first that $|\hat X_{2i}|>1$ for some $i$. Let $u\in X_{2i-2}$ and $v\in X_{2i+2}$, so $uv$ is an edge in $G$. Moreover $\{u,v\}$ is in the strong core. Thus 
$$C_G(\{u,v\})=\bigcup_{\footnotesize \begin{array}{c}j=2i+2 \mod 3,\\\ j\geq 2i+2\end{array}} X_j\ \cup \bigcup_{\footnotesize \begin{array}{c}j=2i-2 \mod 3,\\j\leq 2i-2\end{array}}X_j\cup I_k$$ 

\noindent for $k=2i+1\mod 3$. Hence $C_G(\{u,v\})$ is a clique and so by~\ref{dense_thickening2}, there is a Tihany brace in $G'$, a contradiction. Hence we may assume that $|\hat X_{2i}|=1\ \forall i$. 

Assume that $n\geq 3$ and let $u\in \hat X_2,v\in \hat X_6$.  Then $uv$ is an edge in $G$. Moreover $\{u,v\}$ is in the strong core. Thus

$$C_G(\{u,v\})=\bigcup_{j=0 \mod 3, j\geq 6}X_{j}\cup X_2 \cup R_3 \cup L_5\cup H_3.$$

Hence $C_G(\{u,v\})$ is an antimatching, and by~\ref{lemma:goodclique}, there exists a Tihany triangle in $G'$, a contradiction. It follows that $n\leq 2$.  

Assume now that $n=2$. Let $u\in \hat X_2, v\in L_5$.  Then $uv$ is an edge in $G$ and $C_G(\{u,v\})= X_2\cup R_3\cup L_5\cup H_3$. Thus $G|C(\{u,v\})$ is a perfect anti-matching between $R_3$ and $L_5$. Hence by~\ref{lemma:goodclique}, there is a Tihany triangle in $G'$, a contradiction.

Thus we deduce that $n=1$. Assume that $|R_1|=|L_3|=1$. Let $u\in X_2$ and $v\in R_1\cup L_3$ be a neighbor of $v$. Without loss of generality, we may assume that $v\in L_3$. Since $C_G(\{u,v\})\subseteq X_2\cup L_3\cup H_3$ is a clique, it follows by~\ref{dense_thickening2} that there is a Tihany brace in $G'$, a contradiction. Hence we deduce that $|R_1|=|L_3|>1$. Now, let $u\in R_1$ and $v\in L_3$ be adjacent. By~\ref{non2sub}, we may assume that $G$ is not a 2-non-substantial graph. If follows that there exists $x\in I_2$ such that $x$ is in a triad. Thus $C_G(\{u,v,x\})$ is an antimatching, and by~\ref{lemma:goodclique}, there exists a Tihany clique $K$ in $G'$ with $|K|\leq 4$, a contradiction. This proves~\ref{path_of_triangle}.
\end{proof}
\end{lemma}

\begin{theorem}\label{lk33}
Let $(G,A,B,C)$ be an antiprismatic graph that admit a worn chain decomposition $(G_i,A_i,B_i,C_i)$. Suppose that there exists $k$ such that $(G_k,A_k,B_k,C_k)$ is the line graph of $K_{3,3}$. Let $G'$ be a reduced thickening of $(G,F)$ for some valid $F\in V(G)^2$. If $\chi(G')>\omega(G')$, then there is a Tihany brace in $G'$.
\begin{proof}
Assume not. Let $\{a_j^i\}_{i,j=1,2,3}$ be the vertices of $G_k$ using the standard notation. Let $X_j^i=X_{a_j^i}$ be the clique corresponding to $a_j^i$ in the thickening. Moreover, let $x_i^j \in X_i^j$, $w_i^j=|X_i^j|$ and $k=\chi(G)$.

Since all of the vertices in the thickening of $G_k$ are in triads, $G_k$ is linked to the rest of the graph by a hex-join.

Note that $G\backslash\{x_1^1,x_2^1\}$ is $k$-2 colourable. By~\ref{basic}, it follows that every color class containing a vertex in $X_1^2\cup X_1^3$ must have a vertex in $X_2^1\cup X_3^1$. Hence we deduce that $w_1^2+w_1^3\leq w_2^1+w_3^1-1$ and by symmetry $w_2^2+w_2^3\leq w_1^1+w_3^1-1$. Summing these two inequalities, it follows that  
$$w_1^2+w_1^3+w_2^2+w_2^3< w_2^1+w_1^1+2 w_3^1. $$
A similar inequality can be obtained for all edges $x_i^jx_{i'}^j$. Summing them all, we deduce that $4\sum_{ij} w_i^j < 2\sum_{ij} w_i^j+ 2\sum_{ij} w_i^j,$ a contradiction. This proves~\ref{lk33}
\end{proof}
\end{theorem}

\begin{lemma}\label{3cliques_antiprism}
Let $H$ be a 3-coloured prismatic graph. Let $G$ be a reduced thickening of $(\overline H,F)$ for some valid $F\in V(G)^2$ such that \co. Then there exists a Tihany brace or triangle in $G$.
\begin{proof}
By~\ref{decomp_3cliques}, $H$ admits a worn chain decomposition with all terms in $\mathcal{Q}_0\cup \mathcal{Q}_1\cup \mathcal{Q}_2$. If one term of the decomposition is in $\mathcal{Q}_2$ then by~\ref{path_of_triangle}, it follows that there is a Tihany clique $K$ with $|K|\leq 4$ $G$. If one term of the decomposition is in $\mathcal{Q}_1$, then by~\ref{lk33}, it follows that there is a Tihany brace in $G$. Hence we may assume that all terms are in $\mathcal{Q}_0$. Therefore there are no triads in $G$ and thus by~\ref{alpha2}, it follows that there is a Tihany brace in $G$. This proves~\ref{3cliques_antiprism}. 
\end{proof}
\end{lemma}

We can now prove the main result of this section. 

\begin{lemma}\label{orientable}
Let $H$ be an orientable prismatic graph. Let $G$ be a reduced thickening of $(\overline H,F)$ for some valid $F\subseteq V(G)^2$ such that \co. Then there exists a Tihany clique $K$ in $G$ with $|K|\leq 4$.
\begin{proof}
If $H$ admits a worn chain decomposition with all terms in $\mathcal{Q}_0\cup \mathcal{Q}_1\cup \mathcal{Q}_2$, then by~\ref{3cliques_antiprism}, $G$ admits a Tihany brace or triangle.  Otherwise, by~\ref{decomp_orientable_prism}, $H$ is either not 3-substantial, a cycle of triangles, a ring of five graph, or a mantled $L(K_{3,3})$. 

If $H$ is not 3-substantial, then by~\ref{theorem:not3substantial}, there is a clique $K$ in $G$ with $|K|\leq 4$. If $H$ is a cycle of triangles, then by~\ref{cycleoftriangles}, there is a Tihany brace or triangle in $G$. If $H$ is a ring of five graph, then by~\ref{ringoffive}, there is a Tihany triangle in $G$. Finally, if $H$ is a mantled $L(K_{3,3})$, then by~\ref{mantled}, there is a Tihany brace in $G$. This proves~\ref{orientable}.
\end{proof}
\end{lemma}

\section{Non-orientable Prismatic Graphs}
The definitions needed to understand this section can be found in appendix~\ref{app:non-orientable}. The following is a result from~\cite{claws2}. 

\begin{lemma}\label{twisterandrotator}
Let $G$ be prismatic. Then $G$ is orientable if and only if no induced subgraph of $G$ is a twister or rotator.
\end{lemma}

In the following two lemmas, we study complements of orientable prismatic graphs.  We split our analysis based on whether the graph contains a twister or a rotator as an induced subgraph. 

\begin{lemma}\label{rotator}
Let $H$ be an non-orientable prismatic graph. Assume that there exists $D\subseteq V(H)$ such that $G|D$ is a rotator. Let $G$ be a reduced thickening of $(\overline H,F)$ such that \co~for some valid $F\subseteq V(G)^2$.  Then there exists a Tihany clique $K$ in $G$ with $|K|\leq 5$.
\begin{proof}
Assume not. Let $D=\{v_1,\ldots,v_9\}$ be as in the definition of a rotator. For $i=1,2,3$, let $A_i$ be the set of vertices of $V(H)\backslash D$ that are adjacent to $v_i$. Since $H$ is prismatic and $\{v_1,v_2,v_3\}$ is a triangle, it follows that $A_1\cup A_2\cup A_3=V(H)\backslash D$.

Let $I_1=\{\{5,6\},\{5,9\},\{6,8\},\{8,9\}\}$, $I_2=\{\{4,6\},\{4,9\},\{6,7\},\{7,9\}\}$ and $I_3=\{\{4,5\},\{4,8\},\{5,7\},\{7,8\}\}$. For $i=1,2,3$ and $\{k,l\}\in I_i$, let $A_i^{k,l}$ be the set of vertices of $V(H)\backslash D$ that are complete to $\{v_i,v_k,v_l\}$. Since $\{v_1,v_2,v_3\}$ and $\{v_i,v_{i+3},v_{i+6}\}$ are triangles for $i=1,2,3$ and $H$ is prismatic, we deduce that $A_i=\bigcup_{ \{k,l\}\in I_i} A_i^{k,l}$ for $i=1,2,3$. For $i=1,2,3$ and $\{k,l\}\in I_i$ and since $\{v_1,v_{4},v_{7}\}, \{v_2,v_{5},v_{8}\}, \{v_3,v_{6},v_{9}\}$ are triangles and $H$ is prismatic, it follows that $A_i^{k,l}$ is anticomplete to $v_m$ for all $m\in\{4,5,6,7,8,9\}\backslash\{i,k,l\}$.

Assume that $A_2^{49}$ and $A_3^{48}$ are not empty. Since $H$ is prismatic, we deduce that $A_2^{49}$ is anticomplete to $A_3^{48}$ in $H$. Let $x\in A_2^{49}$ and $y\in A_3^{48}$. Then $C_{\overline H}(\{v_1,v_5,v_6,x,y\}$ is a clique and $\{v_1,v_5,v_6,x,y\}$ is in the strong core. Hence by~\ref{dense_thickening2}, there exists a Tihany clique of size 5 in $G$.

Assume now that $A_2^{49}$ is not empty, but $A_3^{48}$ is empty. Let $x\in A_2^{49}$. Then $C_{\overline H}(\{v_1,v_5,v_6,x\})$ is a clique and $\{v_1,v_5,v_6,x\}$ is in the core. Moreover $\{v_1,v_6,x\}$ is in the strong core. Since $\{v_2,v_5,v_8\}$ is a triad and $v_2$ is in the strong core, it follows that if there exists $E\in F$ with $v_5\in E$, then $E=\{v_5,v_8\}$. But $v_8$ is not adjacent to $v_6$ in $\overline H$. Hence by~\ref{dense_thickening2}, there exists a Tihany clique $K$ of size 4 in $G$.

We may now assume that $A_2^{49}=A_3^{48}=\emptyset$. Since $H$ is prismatic, it follows that $C_{\overline H}(\{v_1,v_5,v_6\})$ is an anti-matching. Moreover $\{v_1,v_5,v_6\}$ is in the core and $v_1$ is in the strong core. For $i=2,3$, since $\{v_i,v_{i+3},v_{i+6}\}$ is a triad and $v_i$ is in the strong core, it follows that if there exists $E\in F$ with $v_{i+3}\in E$, then $E=\{v_{i+3},v_{i+6}\}$. But $v_8$ is not adjacent to $v_6$ and $v_9$ is not adjacent to $v_5$. Hence by~\ref{dense_thickening2}, there exists a Tihany triangle in $G$. This concludes the proof of~\ref{rotator}.
\end{proof}
\end{lemma}

\begin{lemma}\label{lemmatwister}
Let $H$ be a non-orientable prismatic graph. Assume that there exists $W\subseteq V(H)$ such that $H|W$ is a twister. Further, assume that there is no induced rotator in $H$. If $G$ is a reduced thickening of $(\overline H,F)$ such that \co, then there exists a Tihany clique $K$ in $G$ with $|K|\leq 4$.
\begin{proof}
Assume not. Let $W=\{v_1,v_2,\ldots,v_8,u_1,u_2\}$ be as in the definition of a twister. Throughout the proof, all addition is modulo 8. For $i=1,\ldots,8$, let $A_{i,i+1}$ be the set of vertices in $V\backslash W$ that are adjacent to $v_i$ and $v_{i+1}$ and let $B_{i,i+2}$ be the set of vertices in $V\backslash W$ that are adjacent to $v_i$ and $v_{i+2}$. Moreover, let $C\subseteq V\backslash W$ be the set of vertices that are anticomplete to $W$. Since $H$ is prismatic, we deduce that $\bigcup_{i=1}^8 (A_{i,i+1}\cup B_{i,i+2}) \cup C=V\backslash W$. Moreover $A_{i,i+1}$ is complete to $\{v_i,v_{i+1},v_{i+3}, v_{i+6}\}$ and anticomplete to $W\backslash \{v_i,v_{i+1},v_{i+3}, v_{i+6}\}$. Since $H$ is prismatic, it follows also that $B_{i,i+2}$ is complete to $u_{i \mod 2}\}$ and anticomplete to $W\backslash \{v_i,v_{i+2},u_{i\mod 2}\}$. Moreover, $C$ is anticomplete to $\{v_1,v_2,\ldots,v_8\}$. 

\begin{claim}\label{claimtwister1}
There exists $i\in\{1,\ldots,8\}$, such that $A_{i,i+1}$ and $A_{i+3,j+4}$ are either both empty or both non-empty.
\end{claim}
Assume not. By symmetry we may assume that $A_{1,2}$ is not empty and $A_{4,5}$ is empty. Since $A_{1,2}$ is not empty, we deduce that $A_{6,7}$ is empty. Since $A_{4,5}$ and $A_{6,7}$ are empty, it follows that $A_{7,8}$ and $A_{3,4}$ are not empty. Let $x\in A_{7,8}$ and $y\in A_{3,4}$. Then $G|\{v_8,u_1,v_4,x,v_6,v_3,v_7,v_2,y\}$ is a rotator, a contradiction. This proves~\cref{claimtwister1}.

\begin{claim}\label{claimtwister3}
If $A_{i,i+1}$ and $A_{i+3,i+4}$ are both non-empty for some $i\in \{1,\ldots,8\}$, then there exists a Tihany clique of size 5 in $G$.
\end{claim}

Assume that $A_{2,3}$ and $A_{5,6}$ are not empty and let $x\in
A_{2,3}$ and $y\in A_{5,6}$. The anti-neighborhood of
$\{v_1,v_7,u_2,x,y\}$ in $H$ is a stable set. Moreover,
$\{v_1,v_7,u_2,x,y\}$ is in the strong core and hence
by~\ref{dense_thickening2} there is a Tihany clique of size $5$ in
$G$. This proves~\cref{claimtwister3}.

\begin{claim}\label{claimtwister2}
If $A_{i,i+1}$ and $A_{i+3,i+4}$ are both empty for some $i\in \{1,\ldots,8\}$, then there exists a Tihany clique of size 4 in $G$.
\end{claim}
Assume that $A_{2,3}$ and $A_{5,6}$ are both empty. Then the
anti-neighborhood of $\{v_1,v_7,u_2\}$ in $H$ is $A_{8,2}\cup
A_{2,4}\cup A_{4,6}\cup A_{6,8}$ which is a matching. Moreover $u_2$
is in the strong core and $\{v_1,v_7\}$ is in the core. Possibly
$\{v_1,v_5\}$ and $\{v_3,v_7\}$ are in $F$, but $A_{2,8}\cup
A_{2,4}\cup A_{4,6}\cup A_{6,8}\cup \{v_3,v_7\}$ is also an
anti-matching. Hence by~\ref{lemma:goodclique}, there is a Tihany
clique of size $4$ in $G$. This proves~\cref{claimtwister2}.\\

Now by~\cref{claimtwister1}, there exists $i$ such that $A_{i,i+1}$ and $A_{i+3,i+4}$ are either both empty or both non-empty. If $A_{i,i+1}$ and $A_{i+3,i+4}$ are both non-empty, then by~\cref{claimtwister3} there is a Tihany clique of size 5 in $G$. If $A_{i,i+1}$ and $A_{i+3,i+4}$ are both empty, then by~\cref{claimtwister2} there is a Tihany clique of size 4 in $G$. This concludes the proof of~\ref{lemmatwister}.
\end{proof}
\end{lemma}

\begin{lemma}\label{nonorientable}
Let $H$ be a non-orientable prismatic graph. Let $G$ be a reduced thickening of $(\overline H,F)$ for some valid $F\subseteq V(G)^2$ such that \co; then there exists a Tihany clique $K$ in $G$ with $K\leq 5$.
\begin{proof}
By~\ref{twisterandrotator}, it follows that there is an induced twister or an induced rotator in $H$. If there is an induced rotator in $H$, then by~\ref{rotator}, it follows that there is a Tihany clique of size $5$ in $G$. If there is an induced twister and no induced rotator in $H$, then by~\ref{lemmatwister}, it follows that there is a Tihany clique of size 4 in $G$. This proves~\ref{nonorientable}. \end{proof}
\end{lemma}

\section{Three-cliqued Graphs}
In this section we prove Theorem~\ref{main} for those claw-free graphs $G$ for which $V(G)$ can be partitioned into three cliques. The definition of three-cliqued graphs has been given at the start of Section~\ref{sec:orientable}. A list of three-cliqued claw-free graphs that are needed for the statement of the structure theorem can be found in appendix~\ref{app:3cliqued}. We begin with a structure theorem from~\cite{claws5}.

\begin{theorem}\label{clawV_4.1}
Every three-cliqued claw-free graph admits a worn hex-chain into terms each of which is a reduced thickening of a permutation of a member of one of $\mathcal{TC}_1,\ldots,\mathcal{TC}_5$.
\end{theorem}

Let $(G,A,B,C)$ be a three-cliqued graph and $K$ be a clique of $G$. We say that $K$ is \textit{strongly Tihany} if for all three-cliqued graphs $(H,A',B',C')$, $K$ is Tihany in every worn hex-join $(I,A\cup A',B\cup B',C\cup C')$ of $(G,A,B,C)$ and $(H,A',B',C')$ such that $\chi(I)>\omega(I)$.

A clique $K$ is said to be bi-cliqued if exactly two of $K\cap A, K\cap B, K\cap C$ are not empty and every $v \in K$ is in a triad. A clique $K$ is said to be tri-cliqued if $K\cap A, K\cap B, K\cap C$ are all not empty and every $v \in K$ is in a triad.
\begin{lemma}\label{strongly_tihany_1}
Let $K$ be a dense clique in $(G,A_1,A_2,A_3)$. If both $K$ and $\overline C(K)$ are bi-cliqued, then $K$ is strongly Tihany. 
\begin{proof}
Let $(G',A',B',C')$ be a three-cliqued claw-free graph and let $(H,D,E,F)$ be a worn hex-join of $(G,A,B,C)$ and $(G',A',B',C')$.  Then in $H$, $C(K) \cap V(G')$ is a clique that is complete to $C(K) \cap V(G)$.  Hence, by~\ref{W-1}, $K$ is Tihany in $H$ and hence $H$ is strongly Tihany.
\end{proof}
\end{lemma}

\begin{lemma}\label{strongly_tihany_2}
Let $K$ be a dense clique of a three-cliqued graph $(G,A,B,C)$. If $K$ is tri-cliqued, then $K$ is strongly Tihany.
\begin{proof}
Let $(G',A',B',C')$ be a three-cliqued claw-free graph and let $(H,D,E,F)$ be a hex-join of $(G,A,B,C)$ and $(G',A',B',C')$.  Then in $H$, $C_H(K)\cap V(G')=\emptyset$ and thus $C_H(K)$ is a clique in $H$. Hence, by~\ref{W-1}, $K$ is strongly Tihany.
\end{proof}
\end{lemma}

\begin{lemma}\label{type_of_line_graph}
Let $(G,A,B,C)$ be an element of $\mathcal{TC}_1$ and $G'$ be a reduced thickening of $(G,F)$ for some valid $F\subseteq V(G)^2$. Then there is either a strongly Tihany brace or a strongly Tihany triangle in $G'$.
\begin{proof}
Let $H,v_1,v_2,v_3$ be as in the definition of $\mathcal{TC}_1$; so $L(H)=G$. 
Let $V_{12}$ be the set of vertices of $H$ that are adjacent to $v_1$ and $v_2$ but not to $v_3$ and let $V_{13},V_{23}$ be defined similarly. Let $V_{123}$ be the set of vertices complete to $\{v_1,v_2,v_3\}$.

Suppose that $V_{ij}\neq \emptyset$ for some $i,j$.  Then let $v_{ij}\in V_{ij}$, and let $x_i$ be the vertex in $G$ corresponding to the edge $v_{ij}v_i$ in $H$ and $x_j$ be the vertex in $G$ corresponding to the edge $v_{ij}v_j$ in $H$. Then $C_G(\{x_i,x_j\})=\emptyset$, and thus by~\ref{dense_thickening} and~\ref{strongly_tihany_1}, there exists a strongly Tihany brace in $G'$.

So we may assume that $V_{ij}=\emptyset$ for all $i,j$.  Then from the definition of  $\mathcal{TC}_1$, it follows that $V_{123}$ is not empty.  Let $v\in V_{123}$ and let $x_1,x_2,x_3$ be the vertices in $G$ corresponding to the edges $vv_1,vv_2,vv_3$ of $H$, respectively. Then $C_G(\{x_1,x_2,x_3\})=\emptyset$ and hence by~\ref{dense_thickening} and~\ref{strongly_tihany_2}, there exists a strongly Tihany triangle in $G'$. This proves~\ref{type_of_line_graph}.
\end{proof}
\end{lemma}

\begin{lemma}\label{circular interval graphs}
Let $(G,A,B,C)$ be an element of $\mathcal{TC}_2$ and let $(G',A',B',C')$ be a reduced thickening of $(G,F)$ for some valid $F \subseteq V(G)^2$. Then there is either a strongly Tihany brace or a strongly Tihany triangle in $G'$.
\end{lemma}

\begin{proof}
Let $\Sigma,F_1,\ldots,F_k,L_1,L_2,L_3$ be as in the definition of $\mathcal{TC}_2$. Without loss of generality, we may assume that $A$ is not anticomplete to $B$. It follows from the definition of $G$ that there exists $F_i$ such that $F_i\cap A$ and $F_i\cap B$ are both not empty. Let $\{x_k,\ldots,x_l\}=V(H)\cap F_i$ and without loss of generality, we may assume that $\{x_k,\ldots,x_l\}$ are in order on $\Sigma$.

 Let $F_i$ be such that there exists no $j$ with $F_i\subset F_j$. Let $\{x_k,\ldots,x_l\}=V(H)\cap F_i$ and without loss of generality, we may assume that $\{x_k,\ldots,x_l\}$ are in order on $\Sigma$. Since $C(\{x_k,x_l\})=\{x_{k+1},\ldots,x_{l-1}\}$, it follows that $\{x_k,x_l\}$ is dense. If $x_k,x_l$ are the endpoints of $F_i$, it follows by~\ref{basic} and~\ref{dense_thickening} that there is a Tihany brace in $G$. Otherwise, by~\ref{dense_thickening2} there exists a Tihany brace in $G$. This proves~\ref{longcircularinterval}.
\end{proof}

\begin{lemma}\label{near_antiprismatic}
Let $(G,A,B,C)$ be an element of $\mathcal{TC}_3$ and let $(G',A',B',C')$ be a reduced thickening of $(G,F)$ for some valid $F \in V(G)^2$. Then there is either a strongly Tihany brace or a strongly Tihany triangle in $G'$.

\begin{proof}
Let $H,A=\{a_0,a_1,\dots,a_n\},B=\{b_0,b_1,\dots,b_n\}$, $C=\{c_1,\dots,c_n\}$, and $X$ be as in the definition of near-antiprismatic graphs.  Suppose that for some $i$, $a_i,b_i\in V(G)$.  Then since $|C\setminus X|\geq 2$, it follows that there exists $j \neq i$ such that $c_j\in V(G)$.  Now $T=\{a_i,b_i,c_j\}$ is dense and tri-cliqued in $G$, and so by~\ref{dense_thickening} and~\ref{strongly_tihany_2} there is a strongly Tihany triangle in $G'$. 
 
So we may assume that for all $i$, if $a_i \in V(G)$, then $b_i \not\in V(G)$.  Since by definition of $\mathcal{TC}_3$ every vertex is in a triad, it follows that $c_i \in V(G)$ whenever $a_i \in V(G)$.  Now suppose that $a_i,a_j\in V(G)$ for some $i \neq j$.  Then $(\{a_i,a_j\},\{c_i,c_j\})$ is a non-reduced homogeneous pair in $G$.  Hence we may assume that for all $i\neq j$ at most one of $a_i,a_j$ is in $V(G)$.  Let $a_i \in V(G)$,; then for some $j \neq i$ we have $c_j \in V(G)$.  Now $E=\{a_i,c_j\}$ is dense and bi-cliqued. Moreover $\overline C(E)$ is bi-cliqued, hence by~\ref{dense_thickening} and~\ref{strongly_tihany_1}, it follows that $E$ is a strongly Tihany brace in $G'$. This proves~\ref{near_antiprismatic}.
\end{proof}
\end{lemma}

\begin{lemma}\label{sporadic_exceptions}
Let $G$ be an element of $\mathcal{TC}_5$ and $G'$ be a reduced thickening of $(G,F)$ for some valid $F \subseteq V(G)^2$. Then there exists either a brace $E\in V(G')$ that is strongly Tihany or a triangle $T \in V(G')$ that is strongly Tihany in $G'$.

\begin{proof}

First suppose that $G\in \mathcal{TC}_5^1$. Let $H, \{v_1,\dots,v_8\}$ be as in the definition of $\mathcal{TC}_5^1$.  If $v_4\in V(G)$ then $\{v_2,v_4\}$ is dense and bi-cliqued. Moreover $\overline C(\{v_2,v_4\})$ is bi-cliqued and thus by~\ref{dense_thickening} and~\ref{strongly_tihany_1}, there is a strongly Tihany brace in $G'$. If $v_3\in G$, then $\{v_3,v_5\}$ is dense and bi-cliqued. Moreover $\overline C(\{v_3,v_5\})$ is bi-cliqued and so by~\ref{dense_thickening} and~\ref{strongly_tihany_1}, there is a strongly Tihany brace in $G'$. So we may assume that $v_4,v_3 \not\in V(G)$. But then the triangle $T=\{v_1,v_6,v_7\}$ is dense and tri-cliqued and thus by~\ref{dense_thickening} and~\ref{strongly_tihany_2}, there exists a strongly Tihany triangle in $G'$.

We may assume now that $G\in \mathcal{TC}_5^2$. If $v_3\in G$ then $\{v_2,v_3\}$ is dense, bi-cliqued and $\overline C(\{v_2,v_3\})$ is bi-cliqued. Otherwise, $\{v_2,v_4\}$ is dense, bi-cliqued and $\overline C(\{v_2,v_4\})$ is bi-cliqued. In both cases, it follows from~\ref{dense_thickening} and~\ref{strongly_tihany_1} that there exists a strongly Tihany brace in $G'$. This proves~\ref{sporadic_exceptions}.
\end{proof}
\end{lemma}

We are now ready to prove the main result of this section.

\begin{theorem}\label{3cliqued}
Let $G$ be a three-cliqued claw-free graph such that \co. Then $G$ contains either a Tihany brace or a Tihany triangle in $G$.
\end{theorem}
\begin{proof}
By~\ref{clawV_4.1}, there exist $(G_i,A_i,B_i,C_i)$, for $i=1,\ldots,n$, such that the sequence $(G_i,A_i,B_i,C_i)$ $(i=1,\ldots,n)$ is a worn hex-chain for $(G,A,B,C)$ and such that $(G_i,A_i,B_i,C_i)$ is a reduced thickening of a permutation of a member of one of $\mathcal{TC}_1,\ldots,\mathcal{TC}_5$. If there exists $i\in\{1,\ldots,n\}$ such that $(G_i,A_i,B_i,C_i)$ is a reduced thickening of a permutation of a member of $\mathcal{TC}_1$,  $\mathcal{TC}_2$,  $\mathcal{TC}_3$, or $\mathcal{TC}_5$, then by~\ref{type_of_line_graph}, ~\ref{circular interval graphs}, ~\ref{near_antiprismatic}, or \ref{sporadic_exceptions} (respectively), there is a strongly Tihany brace or a strongly Tihany triangle in $G_i$, and thus there is a Tihany brace or a Tihany triangle in $G$.  Thus it follows that $(G_i,A_i,B_i,C_i)$ is a reduced thickening of a member of $\mathcal{TC}_4$ for all $i=1,\ldots,n$. Hence $G$ is a reduced thickening of a three-cliqued antiprismatic graph. By~\ref{3cliques_antiprism}, there exists a Tihany brace or triangle in $G$. This proves~\ref{3cliqued}
\end{proof}

\section{Non-trivial Strip Structures}
In this section we prove~\ref{main} for graphs $G$ that admit non-trivial strip structures and appear in \cite{claws5}.

Let $(J,Z)$ be a strip. We say that $(J,Z)$ is a \emph{line graph strip} if 
$|V(J)|=3$, $|Z|=2$ and $Z$ is complete to $V(J) \setminus Z$.

The following two lemmas appear in \cite{twothirdschi}.

\begin{lemma}\label{1-joins}
Suppose that $G$ admits a nontrivial strip-structure such that $|Z|=1$ for some strip $(J,Z)$ of $(H,\eta)$.  Then either $G$ is a clique or $G$ admits a clique cutset.
\end{lemma}

\begin{lemma}\label{line strips}
Let $G$ be a graph that admits a nontrivial strip-structure $(H,\eta)$ such that for every $F \in E(H)$, the strip of $(H,\eta)$ at $F$ is a line graph strip.  Then $G$ is a line graph.
\end{lemma}

We now use these lemmas to prove the main result of this section.

\begin{theorem}\label{strips}
Let $G$ be a claw-free graph with \co~that is a minimal counterexample to~\ref{main}.  Then $G$ does not admit a nontrivial strip-structure $(H,\eta)$ such that for each strip $(J,Z)$ of $(H,\eta)$, $1 \leq |Z| \leq 2$, and if $|Z|=2$ then either $|V(J)|=3$ and $Z$ is complete to $V(J) \setminus Z$,  or 
$(J,Z)$ is a member of  $\mathcal{Z}_1  \cup \mathcal{Z}_2 \cup 
\mathcal{Z}_3 \cup \mathcal{Z}_4 \cup \mathcal{Z}_5$.
\end{theorem}

\begin{proof}
Suppose that $G$ admits a nontrivial strip-structure $(H,\eta)$ such that for each strip $(J,Z)$ of $(H,\eta)$, $1 \leq |Z| \leq 2$.
Further suppose that $|Z|=1$ for some strip $(J,Z)$.  Then by~\ref{1-joins} either $G$ is a clique or $G$ admits a clique cutset; in the former case 
$G$ does not satisfy \co, and in the latter case~\ref{strips}
follows from~\ref{cliquecutset}.  Hence we may assume that  
$|Z|=2$ for all strips $(J,Z)$.

If all the strips of $(H,\eta)$ are line graph strips, then
by~\ref{line strips}, $G$ is a line graph and the result follows
from \cite{quasi-line}.  So we may assume that some strip $(J_1,Z_1)$
is not a line graph strip.  Let $Z_1=\{a_1,b_1\}$.  Let
$A_1=N_{J_1}(a_1)$, $B_1=N_{J_1}(b_1)$, $A_2=N_G(A_1) \setminus
V(J_1)$, and $B_2=N_G(B_1) \setminus V(J_1)$.  Let $C_1 =
V(J_1) \setminus (A_1 \cup B_1)$ and $C_2= V(G) \setminus (V(J_1)\cup
A_2 \cup B_2)$.  Then $V(G) = A_1 \cup B_1 \cup C_1 \cup A_2 \cup
B_2 \cup C_2$.

\begin{claim}\label{claimstrips}
\noindent If $C_2=\emptyset$ and $A_2=B_2$, then there is a Tihany clique $K$ in $G$ with $|K|\leq 5$. 
\end{claim}
Note that $V(G)=A_1 \cup B_1 \cup C_1 \cup A_2$.  Since $|Z_1|=2$ and $(J_1,Z_1)$ is not a line graph strip, it follows that $(J_1,Z_1)$ is a member of 
$\mathcal{Z}_1  \cup \mathcal{Z}_2 \cup 
\mathcal{Z}_3 \cup \mathcal{Z}_4 \cup \mathcal{Z}_5$.
We consider the cases separately:
\begin{enumerate}
\item If $(J_1,Z_1)$ is a member of $\mathcal{Z}_1$, then $J_1$ is a fuzzy linear interval graph and so $G$ is a fuzzy long circular interval graph and Theorem~\ref{strips} follows from~\cite{quasi-line}.

\item If $(J_1,Z_1)$ is a member of $\mathcal{Z}_2,\mathcal{Z}_3,$ or $\mathcal{Z}_4$.  In all of these cases, $A_1,B_1,$ and $C_1$ are all cliques and so $V(G)$ is the union of three cliques, namely $A_1 \cup A_2, B_1$, and $C_1$.  Hence, by~\ref{3cliqued}, there exists a Tihany clique $K$ with $|K|\leq 5$.

\item If $(J_1,Z_1)$ is a member of $\mathcal{Z}_5$.  Let $v_1,\ldots,v_{12},X,H,H', F$ be as in the definition of $\mathcal{Z}_5$ and for $1 \leq i \leq 12$ let $X_{v_i}$ be as in the definition of a thickening.  Then $A_2$ is complete to $X_{v_1} \cup X_{v_2} \cup X_{v_4} \cup X_{v_5}$.  Let $H''$ be the graph obtained from $H'$
by adding a new vertex $a_2$, adjacent to $v_1,v_2,v_4$ and $v_5$. Then
$H''$ is an antiprismatic graph.  Moreover, no triad of $H''$
contains $v_9$ or $v_{10}$. Thus the
pair $(H',F)$ is antiprismatic, and $G$ is a thickening of $(H',F)$, so~\ref{strips} follows from~\ref{orientable} and~\ref{nonorientable}.

\end{enumerate}
This proves~\cref{claimstrips}.\\

By~\cref{claimstrips}, we may assume that either $C_2 \neq \emptyset$, or
$A_2 \neq B_2$. Suppose that $A_2=B_2$.  Then since $C_2 \neq \emptyset$ it follows that $A_2$ is a clique cutset of $G$ and the result follows from~\ref{cliquecutset}.  Hence, we may assume that $A_2 \neq B_2$ and without loss of generality we may assume that $A_2 \setminus B_2 \neq \emptyset$.  Let $v \in A_2 \setminus B_2$ and let $w \in A_1 \setminus B_1$.  Then $E=\{v,w\}$ is dense and~\ref{strips} follows from~\ref{W-1}.
\end{proof}

\section{Proof of the Main Theorem}
\label{sec:main_proof}We can now prove the main theorem.

\begin{proof}[Proof of~\ref{main}]

Let $G$ be a claw-free graph with \co and suppose that there does not exist a clique $K$ in $G$ with $|K|\leq 5$ such that  $\chi(G \backslash K) > \chi(G) - |K|$.  By~\ref{strips} and~\ref{structure}, it follows that either $G$ is a member of $\mathcal{T}_1\cup\mathcal{T}_2\cup\mathcal{T}_3$ or $V(G)$ is the union of three cliques. By~\ref{icosa}, it follows that $G$ is not a member of $\mathcal{T}_1$. By~\ref{longcircularinterval}, it follows that $G$ is not a member of $\mathcal{T}_2$. By~\ref{orientable} and~\ref{nonorientable}, we deduce that $G$ is not a member $\mathcal{T}_3$. Hence, it follows that $V(G)$ is the union of three cliques. But by~\ref{3cliqued}, it follows that there is a Tihany brace or triangle in $G$, a contradiction. This proves~\ref{main}.
\end{proof}

\label{totalpag}

\newpage
\appendix
\section{Orientable prismatic graphs}
\label{app:orientable}
\begin{itemize}
\item $\mathcal Q_0$ is the class of all 3-coloured graphs $(G,A,B,C)$ such that $G$ has no triangle.
\item $\mathcal Q_1$ is the class of all 3-coloured graphs $(G,A,B,C)$ where $G$ is isomorphic to the line graph of $K_{3,3}$.
\item $\mathcal Q_2$ is the class of all canonically-coloured path of triangles graphs.
\end{itemize}
\begin{itemize}
\item \textbf{Path of triangles.} A graph $G$ is a \textit{path of triangles} graph if for some integer $n \geq 1$ there are pairwise disjoint stable subsets $X_1,\ldots,X_{2n+1}$ of $V(G)$ with union $V (G)$, satisfying the following conditions (P1)-(P7).
\begin{itemize}
\item[(P1)] For $1 \leq i \leq n$, there is a nonempty subset $\hat X_{2i} \subseteq X_{2i}; |\hat X_2| = |\hat X_{2n}| = 1$, and for $0 < i < n$, at least one of $\hat X_{2i},\hat X_{2i+2}$ has cardinality 1.
\item[(P2)] For $1 \leq i < j \leq 2n + 1$
\begin{enumerate}
\item[(1)] if $j-i = 2$ modulo 3 and there exist $u \in X_i$ and $v \in X_j$, nonadjacent, then either $i, j$ are
odd and $j = i + 2$, or $i, j$ are even and $u \notin \hat X_i$ and $v \notin \hat X_j$;
\item[(2)] if $j - i \neq 2$ modulo 3 then either $j = i + 1$ or $X_i$ is anticomplete to $X_j$.
\end{enumerate}
\item[(P3)] For $1 \leq i \leq n+ 1$, $X_{2i-1}$ is the union of three pairwise disjoint sets $L_{2i-1},M{2i-1},R{2i-1}$, where
$L_1 = M_1 = M_{2n+1} = R_{2n+1} = \emptyset$.
\item[(P4)] If $R_1 = \emptyset$ then $n \geq 2$ and $|\hat X_4| > 1$, and if $L_{2n+1} = \emptyset$ then $n \geq 2$ and $|\hat X_{2n-2}| > 1$.
\item[(P5)] For $1 \leq i \leq n$, $X_{2i}$ is anticomplete to $L_{2i-1}\cup R_{2i+1}$; $X_{2i}\backslash \hat X_{2i}$ is anticomplete to $M_{2i-1}\cup M_{2i+1}$; and every vertex in $X_{2i} \backslash \hat X_{2i}$ is adjacent to exactly one end of every edge between $R_{2i-1}$ and
$L_{2i+1}$.
\item[(P6)] For $1 \leq i \leq n$, if $| \hat X_{2i}| = 1$, then
\begin{enumerate}
\item[(1)] $R_{2i-1},L_{2i+1}$ are matched, and every edge between $M_{2i-1} \cup R_{2i-1}$ and $L_{2i+1}\cup  M_{2i+1}$ is
between $R_{2i-1}$ and $L_{2i+1}$;
\item[(2)] the vertex in $\hat X_{2i}$ is complete to $R_{2i-1} \cup M_{2i-1} \cup L_{2i+1} \cup M_{2i+1}$;
\item[(3)] $L_{2i-1}$ is complete to $X_{2i+1}$ and $X_{2i-1}$ is complete to $R_{2i+1}$
\item[(4)] if $i > 1$ then $M_{2i-1}, \hat X_{2i-2}$ are matched, and if $i < n$ then $M_{2i+1},\hat X_{2i+2}$ are matched.
\end{enumerate}
\item[(P7)] For $1 < i < n$, if $|\hat X_{2i}| > 1$ then
\begin{itemize}
\item[(1)] $R_{2i-1} = L_{2i+1} = \emptyset$;
\item[(2)] if $u \in X_{2i-1}$ and $v \in X_{2i+1}$, then $u, v$ are nonadjacent if and only if they have the same
neighbour in $\hat X_{2i}$.
\end{itemize}
\end{itemize}
Let $A_k=\bigcup(X_i:1\leq i \leq 2n+1 \makebox{ and } i=k\mod 3)\ (k=0,1,2)$. Then $(G,A_1,A_2,A_3)$ is a \textit{canonically-coloured path of triangles graphs}.
\item \textbf{Cycle of triangles.} A graph $G$ is a \textit{cycle of triangles} graph if for some integer $n \geq 5$ with $n = 2$ modulo
3, there are pairwise disjoint stable subsets $X_1,\ldots,X_{2n}$ of $V(G)$ with union $V(G)$, satisfying the
following conditions (C1)-(C6) (reading subscripts modulo $2n$):
\begin{itemize}
\item[(C1)] For $1 \leq i \leq n$, there is a nonempty subset $\hat X_{2i} \subseteq X_{2i}$, and at least one of $\hat X_{2i}, \hat X_{2i+2}$ has cardinality 1.
\item[(C2)] For $i \in \{1,\ldots,2n\}$ and all $k$ with $2 \leq k \leq 2n - 2$, let $j \in \{1,\ldots, 2n\}$ with $j = i + k$ modulo
2n:
\begin{itemize}
\item[(1)] if $k = 2$ modulo 3 and there exist $u \in X_{i}$ and $v \in X_j$ , nonadjacent, then either $i, j$ are
odd and $k \in \{2, 2n - 2\}$, or $i, j$ are even and $u \notin \hat X_i$ and $v \notin \hat X_j$ ;
\item[(2)] if $k \neq 2$ modulo 3 then $X_i$ is anticomplete to $X_j$.
\end{itemize}
(Note that $k = 2$ modulo 3 if and only if $2n - k = 2$ modulo 3, so these statements are
symmetric between $i$ and $j$.)
\item[(C3)] For $1 \leq i \leq n + 1$, $X_{2i-1}$ is the union of three pairwise disjoint sets $L_{2i-1},M_{2i-1},R_{2i-1}$.
\item[(C4)] For $1 \leq  i \leq  n$, $X_{2i}$ is anticomplete to $L_{2i-1}\cup R_{2i+1}$; $X_{2i} \backslash \hat X_{2i}$ is anticomplete to $M_{2i-1} \cup M_{2i+1}$; and every vertex in $X_{2i} \backslash \hat X_{2i}$ is adjacent to exactly one end of every edge between $R_{2i-1}$ and
$L_{2i+1}$.
\item[(C5)] For $1 \leq  i \leq  n$, if $| \hat X_{2i}| = 1$, then
\begin{itemize}
\item[(1)] $R_{2i-1},L_{2i+1}$ are matched, and every edge between $M_{2i-1} \cup  R_{2i-1}$ and $L_{2i+1} \cup  M_{2i+1}$ is
between $R_{2i-1}$ and $L_{2i+1}$;
\item[(2)] the vertex in $\hat X_{2i}$ is complete to $R_{2i-1} \cup M_{2i-1} \cup  L_{2i+1} \cup M_{2i+1}$;
\item[(3)] $L_{2i-1}$ is complete to $X_{2i+1}$ and $X_{2i-1}$ is complete to $R_{2i+1}$
\item[(4)] $M_{2i-1}, \hat X_{2i-2}$ are matched and $M_{2i+1}, \hat X_{2i+2}$ are matched.
\end{itemize}
\item[(C6)] For $1 \leq  i \leq  n$, if $|\hat X_{2i}| > 1$ then
\begin{itemize}
\item[(1)] $R_{2i-1} = L_{2i+1} = \emptyset$;
\item[(2)] if $u \in X_{2i-1}$ and $v \in X_{2i+1}$, then $u,v$ are nonadjacent if and only if they have the same
neighbour in $\hat X_{2i}$.
\end{itemize}
\end{itemize}
\item \textbf{Ring of five.} Let $G$ be a graph with $V(G)$ the union of the disjoint sets $W = \{a_1,\ldots,a_5, b_1,\ldots,b_5\}$ and $V_0, V_1,\ldots,V_5$.
Let adjacency be as follows (reading subscripts modulo 5). For $1 \leq i \leq 5$, $\{a_i, a_{i+1}; b_{i+3}\}$ is a triangle,
and $a_i$ is adjacent to $b_i$; $V_0$ is complete to $\{b_1,\ldots,b_5\}$ and anticomplete to $\{a_1,\ldots,a_5\}$; $V_0, V_1,\ldots, V_5$
are all stable; for $i = 1, \ldots,5$, $V_i$ is complete to $\{a_{i-1}, b_i, a_{i+1}\}$ and anticomplete to the remainder of
$W$; $V_0$ is anticomplete to $V_1\cup\cdots\cup V_5$; for $1 \leq i \leq 5$ $V_i$ is anticomplete to $V_{i+2}$; and the adjacency
between $V_i, V_{i+1}$ is arbitrary. We call such a graph a \textit{ring of five}.

\item \textbf{Mantled $L(K_{3,3})$.} Let $G$ be a graph with $V(G)$ the union of seven sets
$$W = \{a_i^j : 1 \leq i, j \leq 3\}, V^1, V^2, V^3, V_1, V_2, V_3,$$
with adjacency as follows. For $1 \leq i, j, i', j' \leq 3$, $a_i^j$ and $a_{i'}^{j'}$ are adjacent if and only if $i'\neq i$ and $j' \neq j$.
For $i = 1, 2, 3, V^i, V_i$ are stable; $V^i$ is complete to $\{a_i^1, a_i^2, a_i^3\}$, and anticomplete to the remainder of $W$; and $V_i$ is complete to $\{a_1^i, a_2^i, a_3^i\}$ and anticomplete to the remainder of $W$. Moreover, $V^1\cup V^2\cup V^3$
is anticomplete to $V_1 \cup V_2 \cup V_3$, and there is no triangle included in $V^1 \cup V^2 \cup V^3$ or in $V_1 \cup V_2 \cup V_3$.
We call such a graph G a \textit{mantled $L(K_{3,3})$}.
\end{itemize}
\section{Non-orientable prismatic graphs}
\label{app:non-orientable}
\begin{itemize}
\item \textbf{A rotator}. Let $G$ have nine vertices $v_1, v_2,\ldots, v_9$, where $\{v_1, v_2, v_3\}$ is a triangle, $\{v_4, v_5, v_6\}$ is complete to $\{v_7, v_8, v_9\}$, and for $i=1,2,3$, $v_i$ is adjacent to $v_{i+3}$, $v_{i+6}$, and there are no other edges. We call $G$ a \textit{rotator}.
\item \textbf{A twister}. Let G have ten vertices $u_1, u_2, v_1,\ldots, v_8$ , where $u_1, u_2$ are adjacent, for $i=1,\ldots, 8$ $v_i$ is adjacent to $v_{i-1}, v_{i+1}, v_{i+4}$ (reading subscripts modulo $8$), and for $i= 1,2$, $u_i$ is adjacent to $v_i, v_{i+2}, v_{i+4}, v_{i+6}$, and there are no other edges. We call $G$ a \textit{twister} and $u_1$, $u_2$ is the \textit{axis} of the twister.
\end{itemize}
\section{Three-cliqued graphs}
\label{app:3cliqued}

\begin{itemize}
\item \textbf{A type of line trigraph}. Let $v_1, v_2, v_3$ be distinct nonadjacent vertices of a graph $H$, such that every edge of $H$ is incident with one of $v_1, v_2, v_3$. Let $v_1, v_2, v_3$ all have degree at least three, and let all other vertices of $H$ have degree at least one. Moreover, for all distinct $i,j \in \{1,2,3\}$, let there be at most one vertex different from $v_1, v_2, v_3$ that is adjacent to $v_i$ and not to $v_j$ in $H$. Let $A,B,C$ be the sets of edges of $H$ incident with $v_1, v_2, v_3$ respectively, and let $G$ be a line trigraph of $H$. Then $(G, A,B,C)$ is a three-cliqued claw-free trigraph; let $\mathcal{TC}_1$ be the class of all such three-cliqued trigraphs such that every vertex is in a triad.

\item \textbf{Long circular interval trigraphs}. Let $G$ be a long circular interval trigraph, and let $\Sigma$ be a circle with $V(G)\subseteq \Sigma$, and $F_1,\ldots,F_k\subseteq \Sigma$, as in the definition of long circular interval trigraph. By a \textit{line} we mean either a subset $X\subseteq V(G)$ with $|X|\leq 1$, or a subset of some $F_i$ homeomorphic to the closed unit interval, with both end-points in $V (G)$. Let $L_1,L_2,L_3$ be pairwise disjoint lines with $V(G)\subseteq L_1\cup L_2 \cup L_3$; then $(G, V (G)\cap L_1, V(G)\cap L_2, V (G)\cap L_3)$ is a three-cliqued claw-free trigraph. We denote by $\mathcal{TC}_2$ the class of such three-cliqued trigraphs with the additional property that every vertex is in a triad.

\item \textbf{Near-antiprismatic trigraphs}. Let $H$ be a near-antiprismatic trigraph, and let $A,B,C,X$
be as in the deffnition of near-antiprismatic trigraph. Let $A' = A \backslash X$ and define $B',C'$
similarly; then $(H,A',B',C')$ is a three-cliqued claw-free trigraph. We denote by $\mathcal{TC}_3$ the class
of all three-cliqued trigraphs with the additional property that every vertex is in a triad.

\item \textbf{Antiprismatic trigraphs}. Let $G$ be an antiprismatic trigraph and let $A,B,C$ be a partition
of $V(G)$ into three strong cliques; then $(G,A,B,C)$ is a three-cliqued claw-free trigraph. We
denote the class of all such three-cliqued trigraphs by $\mathcal{TC}_4$. (In \cite{claws1} Chudnovsky and Seymour described explicitly all
three-cliqued antiprismatic graphs, and their "changeable" edges; and this therefore provides
a description of the three-cliqued antiprismatic trigraphs.) Note that in this case there may be
vertices that are in no triads.

\item \textbf{Sporadic exceptions.}
\begin{itemize}
\item Let $H$ be the trigraph with vertex set $\{v_1,\ldots,v_8\}$ and adjacency as follows: $v_i, v_j$ are strongly adjacent for $1 \leq i < j \leq 6$ with $j - i \leq 2$; the pairs $v_1v_5$ and $v_2v_6$ are strongly antiadjacent; $v_1, v_6, v_7$ are pairwise strongly adjacent, and $v_7$ is strongly antiadjacent to $v_2, v_3, v_4, v_5$; $v_7, v_8$ are strongly adjacent, and $v_8$ is strongly antiadjacent to  $v_1,\ldots,v_6$; the pairs $v_1v_4$ and $v_3v_6$ are semiadjacent, and $v_2$ is antiadjacent to $v_5$. Let $A = \{v_1, v_2, v_3\},B = \{v_4, v_5, v_6\}$ and $C = \{v_7, v_8\}$. Let $X\subseteq \{v_3, v_4\}$; then $(H \backslash X,A \backslash X,B \backslash X,C)$ is a three-cliqued claw-free trigraph, and all its vertices are in triads.
\item Let $H$ be the trigraph with vertex set $\{v_1,\ldots, v_9\}$, and adjacency as follows: the sets $A = \{v_1, v_2\}$, $B = \{v_3, v_4, v_5, v_6, v_9\}$ and $C = \{v_7, v_8\}$ are strong cliques; $v_9$ is strongly adjacent to $v_1, v_8$ and strongly antiadjacent to $v_2, v_7$; $v_1$ is strongly antiadjacent to $v_4, v_5, v_6, v_7$, semiadjacent to $v_3$ and strongly adjacent to $v_8$; $v_2$ is strongly antiadjacent to $v_5, v_6, v_7, v_8$ and strongly adjacent to $v_3$; $v_3, v_4$ are strongly antiadjacent to $v_7, v_8$; $v_5$ is strongly antiadjacent to $v_8$; $v_6$ is semiadjacent to $v_8$ and strongly adjacent to $v_7$; and the adjacency between the pairs $v_2v_4$ and $v_5v_7$ is arbitrary. Let $X \subseteq  \{v_3, v_4, v_5, v_6\}$, such that 
\begin{itemize}
\item $v_2$ is not strongly anticomplete to $\{v_3, v_4\}\backslash X$
\item $v_7$ is not strongly anticomplete to $\{v_5, v_6\} \backslash X$
\item if $v_4, v_5 \notin X$ then $v_2$ is adjacent to $v_4$ and $v_5$ is adjacent to $v_7$. 
\end{itemize}
Then $(H \backslash X, A,B \backslash X,C)$ is a three-cliqued claw-free trigraph. 
\end{itemize}
We denote by $\mathcal{TC}_5$ the class of such three-cliqued trigraphs (given by one of these two constructions) with the additional property that every vertex is in a triad. 

\end{itemize}

\section{Strips}\label{appendix:strips}

  \newcounter{Lcount}
  \begin{list}{$\mathcal{Z}_\arabic{Lcount}$:}
    {\usecounter{Lcount}
    \setlength{\rightmargin}{\leftmargin}}
 \item Let $H$ be a graph with vertex set $\{v_1,\ldots,v_n\}$, such that for $1 \leq i < j <k \leq n$, if $v_i,v_k$ are adjacent then $v_j$ is adjacent to both $v_i,v_k$.  Let $n \geq 2$, let $v_1,v_n$ be nonadjacent, and let there be no vertex adjacent to both $v_1$ and $v_n$.  Let $F' \subseteq V(H)^2$ be the set of pairs $\{v_i,v_j\}$ such that $i<j$, $v_i \neq v_1$ and $v_j \neq v_n$, $v_i$ is nonadjacent to $v_{j+1}$, and $v_j$ is nonadjacent to $v_{i-1}$.  Furthermore, let $F \subseteq F'$ such that every vertex of $H$ appears in at most one member of $F$.  Then $G$ is a \emph{fuzzy linear interval graph} if for some $H$ and $F$, $G$ is a thickening of $(H,F)$ with $|X_{v_1}|=|X_{v_n}|=1$.  Let $X_{v_1}=\{u_1\}$, $X_{v_n}=\{u_n\}$, and $Z=\{u_1,u_n\}$. $\mathcal{Z}_1$ is the set of all pairs $(G,Z)$.

\item Let $n \geq 2$.  Construct a graph $H$ as follows.  Its vertex set is the disjoint union of three sets $A,B,C$, where $|A|=|B|=n+1$ and $|C|=n$, say $A=\{a_0,a_1,\ldots,a_n\},B=\{b_0,b_1,\ldots,b_n\}$, and $C=\{c_1,\ldots,c_n\}$.  Adjacency is as follows.  $A,B,C$ are cliques.  For $0 \leq i,j \leq n$ with $(i,j) \neq (0,0)$, let $a_i,b_j$ be adjacent if and only if $i=j$, and for $1 \leq i \leq n$ and $0 \leq j \leq n$, let $c_i$ be adjacent to $a_j,b_j$ if and only if $i \neq j \neq 0$.  All other pairs not specified so far are nonadjacent.  Now let $X \subseteq A \cup B \cup C \setminus \{a_0,b_0\}$ with $|C \setminus X| \geq 2$.  Let $H'=H \setminus X$ and let $G$ be a thickening of $(H',F)$ with $|X_{a_0}|=|X_{b_0}|=1$  and $F \subseteq V(H')^2$ (we will not specify the possibilities for the set $F$ because they are technical and we will not need them in our proof).  Let $X_{a_0}=\{a_0'\}$, $X_{b_0}=\{b_0'\}$, and $Z=\{a_0',b_0'\}$. $\mathcal{Z}_2$ is the set of all pairs $(G,Z)$.

\item Let $H$ be a graph, and let $h_1\d h_2\d h_3\d h_4\d h_5$ be the vertices of a path of $H$ in order, such that $h_1,h_5$ both have degree one in $H$, and every edge of $H$ is incident with one of $h_2,h_3,h_4$.  Let $H'$ be obtained from the line graph of $H$ by making the edges $h_2h_3$ and $h_3h_4$ of $H$ (vertices of $H'$) nonadjacent.  Let $F \subseteq \{\{h_2h_3,h_3h_4\}\}$ and let $G$ be a thickening of $(H',F)$ with $|X_{h_1h_2}|=|X_{h_4h_5}|=1$.  Let $X_{h_1h_2}=\{u\}$, $X_{h_4h_5}=\{v\}$, and $Z=\{u,v\}$. $\mathcal{Z}_3$ is the set of all pairs $(G,Z)$.

\item Let $H$ be the graph with vertex set $\{a_0,a_1,a_2,b_0,b_1,b_2,b_3,c_1,c_2\}$ and adjacency as follows: $\{a_0,a_1,a_2\}, \{b_0,b_1,b_2,b_3\}, \{a_2,c_1,c_2\},$ and $\{a_1,b_1,c_2\}$ are cliques; $b_2,c_1$ are adjacent; and all other pairs are nonadjacent.  Let $F=\{\{b_2,c_2\},\{b_3,c_1\}\}$ and let $G$ be a thickening of $(H,F)$ with $|X_{a_0}|=|X_{b_0}|=1$.  Let $X_{a_0}=\{a_0'\}$, $X_{b_0}=\{b_0'\}$, and $Z=\{a_0',b_0'\}$. $\mathcal{Z}_4$ is the set of all pairs $(G,Z)$.

\item Let $H$ be the graph with vertex set $\{v_1,\ldots,v_{12}\}$, and with adjacency as follows.  $v_1\d \cdots \d v_6\d v_1$ is an induced cycle in $G$ of length 6.  Next, $v_7$ is adjacent to $v_1,v_2$; $v_8$ is adjacent to $v_4,v_5$; $v_9$ is adjacent to $v_6,v_1,v_2,v_3$; $v_{10}$ is adjacent to $v_3,v_4,v_5,v_6,v_9$; $v_{11}$ is adjacent to $v_3,v_4,v_6,v_1,v_9,v_{10}$; and $v_{12}$ is adjacent to $v_2,v_3,v_5,v_6,v_9,v_{10}$.  No other pairs are adjacent.  Let $H'$ be a graph isomorphic to $H \setminus X$ for some $X \subseteq \{v_{11},v_{12}\}$ and let $F \subseteq \{\{v_9,v_{10}\}\}$.  Let $G$ be a thickening of $(H',F)$ with $|X_{a_0}|=|X_{b_0}|=1$. Let $X_{v_7}=\{v_7'\}$, $X_{v_8}=\{v_8'\}$, and $Z=\{v_7',v_8'\}$. $\mathcal{Z}_5$ is the set of all pairs $(G,Z)$.

\end{list}
\end{document}